\documentclass[11pt,a4paper,notitlepage]{article}
\usepackage{amsmath}
\usepackage[affil-it]{authblk}
\usepackage{amsthm}
\usepackage{amssymb}
\usepackage{enumerate}
\usepackage[page]{appendix}
\usepackage{geometry}
\usepackage{tabularx}
\usepackage{multirow}
\newtheorem{Lemma}      {Lemma} [section]
\newtheorem{Theorem}    [Lemma] {Theorem}
\newtheorem{Corollary}  [Lemma] {Corollary}
\newtheorem{Proposition}[Lemma] {Proposition}

\theoremstyle{definition}
\newtheorem{Definition} [Lemma] {Definition}
\newtheorem{Remark} [Lemma] {Remark}

\numberwithin{equation}{section}
\newcommand{\core}{\mathop{\mathrm{core}}}
\newcommand{\Ker}{\mathop{\mathrm{Ker}}}
\newcommand{\Stab}{\mathop{\mathrm{Stab}}}

\newcommand{\ab}{{a_{ab}}}
\newcommand{\nab}{{c_{nonab}}}
\newcommand{\anab}{{a_{nonab}}}
\newcommand{\bl}{{bl}}
\newcommand{\ct}{{c_{trans}}}
\newcommand{\cp}{{c_{prim}}}
\newcommand{\Sym}{\mathop{\mathrm{Sym}}}
\newcommand{\Alt}{\mathop{\mathrm{Alt}}}
\newcommand{\Gal}{\mathop{\mathrm{Gal}}}

\newcommand{\Aut}{\mathop{\mathrm{Aut}}}
\newcommand{\Out}{\mathop{\mathrm{Out}}}
\newcommand{\Soc}{\mathop{\mathrm{Soc}}}

\newcommand{\lpp}{\mathop{\mathrm{lpp}}}

\newcommand{\di}{{d_{I}}}

\newcommand{\ws}{{\widetilde{\omega}}}

\numberwithin{equation}{section}
\begin{document}
\title{Invariable generation of permutation and linear groups}
\author{Gareth Tracey%
\thanks{Electronic address: \texttt{G.M.Tracey@bath.ac.uk}}} 
\affil{Department of Mathematical Sciences, University of Bath,\\Bath BA2 7AY, United Kingdom}
\date{January 30, 2017}
\maketitle
\begin{abstract} A subset $\left\{x_{1},x_{2},\hdots,x_{d}\right\}$ of a group $G$ \emph{invariably generates} $G$ if $\left\{x_{1}^{g_{1}},x_{2}^{g_{2}},\hdots,x_{d}^{g_{d}}\right\}$ generates $G$ for every $d$-tuple $(g_{1},g_{2}\hdots,g_{d})\in G^{d}$. We prove that a finite completely reducible linear group of dimension $n$ can be invariably generated by $\left\lfloor \frac{3n}{2}\right\rfloor$ elements. We also prove tighter bounds when the field in question has order $2$ or $3$. Finally, we prove that a transitive [respectively primitive] permutation group of degree $n\geq 2$ [resp. $n\geq 3$] can be invariably generated by $O\left(\frac{n}{\sqrt{\log{n}}}\right)$ [resp. $O\left(\frac{\log{n}}{\sqrt{\log{\log{n}}}}\right)$] elements.\end{abstract}

\section{Introduction}
There is a continually growing body of literature which broadly concerns the various generation properties in finite groups. This has involved the analysis of many ``generation type" group theoretic invariants, such as the minimal size $d(G)$ of a generating set for the group $G$. In this paper, we study a related invariant: the minimal size of an \emph{invariable} generating set. 

\begin{Definition}\label{InvDef} Let $G$ be a group.\begin{enumerate}[(a)]
\item We say that a subset $\left\{x_{1},x_{2},\hdots,x_{d}\right\}$ of $G$ \emph{invariably generates} $G$ if $\left\{x_{1}^{g_{1}},x_{2}^{g_{2}},\hdots,x_{d}^{g_{d}}\right\}$ generates $G$ for every $d$-tuple $(g_{1},g_{2}\hdots,g_{d})\in G^{d}$.
\item Suppose that $G$ is finite. Define $\di(G)$ to be the smallest size of an invariable generating set for $G$.\end{enumerate} \end{Definition}

\subsection{Warnings}
There are a couple of warnings that should be pointed out here. Firstly, there exist infinite groups $G$ in which $\bigcup_{g\in G} H^g=G$ for some proper subgroup $H$ of $G$. In this case $G$ does not even have an invariable generating set, so $\di(G)$ is certainly not well-defined. Thus, the requirement that $G$ is finite in Definition \ref{InvDef} Part (b) really is necessary (of course one could also weaken this to ``finitely invariably generated"). 

Secondly, one may be tempted to study an even stronger generation property, by requiring that $\langle \left\{x_{1}^{g_{1}},x_{2}^{g_{2}},\hdots,x_{d}^{g_{d}}\right\}\rangle=G$ for every $d$-tuple $(g_{1},g_{2}\hdots,g_{d})\in \Aut(G)^{d}$. However, an arbitrary finite group $G$ may not even contain such a generating set. For example, when $G$ is elementary abelian of order $p^a$, for some prime $p$, then $\Aut(G)\cong GL_a(p)$ acts transitively on the non-identity elements of $G$. 

\subsection{History, motivation and main results} 
The notion of invariable generation was first discussed by B.L. van der Waerden in 1934 \cite{VDW}. Motivated by the problem of computing Galois groups, van der Waerden asked about probabilistic invariable generation in the case $G=\Sym(n)$. For more information about this direction see also \cite{Dix92}.

Suppose that a subset $X:=\{x_1,\hdots,x_d\}$ of the finite group $G$ fails to invariable generate $G$. Then $X$ is contained in the union $\bigcup_{g\in G}M^g$ of conjugates of a maximal subgroup $M$ of $G$. In other words, no element of $X$ acts fixed point freely in the action of $G$ on the set of (right) cosets of $M$ in $G$. Thus, the study of invariable generation is closely related to the theory of derangements in transitive permutation groups. This direction has seen a lot of recent attention, particularly in the case when the group in question is a non-abelian simple group: see \cite{LuPyb}, \cite{FG1}, \cite{FG2}, \cite{FG3}, \cite{FG4}, \cite{peres}, and \cite{EFG}. 

A more general analysis of invariable generation in finite groups was undertaken by Kantor, Lubotzky and Shalev in \cite{KLS}. This is, as far as we know, where the notation $\di(G)$ first appeared. Among many interesting results, they showed that while it is clear that $\di(G)\ge d(G)$, we have $\di(G)=d(G)$ when $G$ is nilpotent \cite[Proposition 2.4]{KLS}, but that $\di(G)-d(G)$ can be arbitrarily large in general \cite[Propostion 2.5]{KLS}. In fact, this is even true if we restrict to the case when $G$ is soluble (see \cite[Corollary 10]{DetLuc2}). Thus, if one fixes a class of finite groups $\mathcal{C}$, it is an interesting (and open) problem to determine if bounds on $d(G)$ carry over to comparable bounds on $\di(G)$. For instance, it has been proven independently in \cite{GurMalle} and \cite{KLS} that $\di(G)=d(G)=2$ when $G$ is a nonabelian finite simple group.

In this paper, we investigate the ``naturally occurring finite groups", namely the permutation and linear groups. A. Lucchini and E. Detomi (see Theorem \ref{DetLucHalfn}) have proved that the ``McIver-Neumann half $n$ bound", which states that \emph{$d(G)\le \frac{n}{2}$ whenever $G$ is a permutation group of degree $n$, and ($G$, $n$)$\neq$ ($S_3$, $3$)}, holds when one replaces $d$ by $\di$. Our first two main results deal with the case when $G$ is transitive and primitive, respectively.
\begin{Theorem}\label{TransInvTheorem} There exists an absolute constant $\ct$ such that $$d(G)\le \frac{\ct n}{\sqrt{\log{n}}}$$
whenever $G$ is a transitive permutation group of degree $n\ge 2$.\end{Theorem}
\begin{Theorem}\label{PrimInvTheorem} There exists an absolute constant $\cp$ such that $$d(G)\le \frac{\cp\log{n}}{\sqrt{\log{\log{n}}}}$$
whenever $G$ is a primitive permutation group of degree $n\ge 3$.\end{Theorem}

When $\di$ is replaced by $d$, Theorem \ref{TransInvTheorem} is \cite[Theorem 1]{LucMenMor}, while Theorem \ref{PrimInvTheorem} is \cite[Theorem C]{LucMenMor2}.

We now move on to linear groups. For $n$ even, we follow \cite{derek} and denote by $B_{n}$ the completely reducible group $B_{n}:=3^{n/2}:2\le GL_{n}(2)$, such that $Z(B_{n})=1$, and $B_{n}\le GL_{2}(2)^{n/2}$ acts completely reducibly on a direct sum of $2$-dimensional submodules.

We also require the following definitions.
\begin{Definition}\label{ImprimitiveDef} Let $\mathbb{F}$ be a field. An irreducible subgroup $R$ of $GL_{n}(\mathbb{F})$ is called \emph{imprimitive} if the natural $R$-module $V\cong \mathbb{F}^n$ has a direct sum decomposition $V=W_{1}\oplus W_{2}\oplus\hdots\oplus W_{r}$, where $r>1$ and $R$ acts on the set $\left\{W_{1},W_{2},\hdots,W_{r}\right\}$. If no such decomposition exists, then $R$ is \emph{primitive}. If every normal subgroup of $R$ is homogeneous, then we say that $R$ is \emph{quasiprimitive}, while if every characteristic subgroup of $R$ is homogeneous then $R$ is said to be \emph{weakly quasiprimitive}.\end{Definition}

A primitive group $R\le GL_{n}(\mathbb{F})$ is both quasiprimitive and weakly quasiprimitive, since a decomposition of $V$ into homogeneous components for a normal or characteristic subgroup of $R$ would yield an imprimitivity decomposition for $V$ as given in Definition \ref{ImprimitiveDef}. The precise statement of our theorem can now be given as follows.  
\begin{Theorem}\label{CompRedTheorem} Let $\mathbb{F}$ be a field.\begin{enumerate}
\item Let $G\le GL_{n}(\mathbb{F})$ be finite and completely reducible. Then $\di(G)\le 3n/2$. Furthermore, if $|\mathbb{F}|=2$ then $\di(G)\le n/2$, unless $G\cong B_{n}$ as defined above, in which case we have $\di(G)=n/2+1$; or $G\cong Sp_{4}(2)\cong S_{6}$, in which case we have $\di(G)=3$. Also, if $|\mathbb{F}|=3$ then $\di(G)\le n$.
\item Let $R\le GL_{n}(\mathbb{F})$ be finite and weakly quasiprimitive. Also, let $Z=R\cap Z(GL_{n}(\mathbb{F}))$, and let $H$ be a subnormal subgroup of $R$. Then $\di(HZ/Z)\le 2\log{n}$, unless $R=H=Sp_{4}(2)\cong S_{6}$, in which case we have $\di(HZ/Z)=3$.\end{enumerate}\end{Theorem}

For the case $|\mathbb{F}|=2$ in Part 1 of Theorem \ref{CompRedTheorem}, and the exceptional case $G\cong B_{n}$, compare Theorem \ref{DetLucHalfn} and its exceptional case, which, as mentioned above, is the corresponding result for permutation groups. 

If one replaces $\di$ by $d$, then the first statement in Part 1 of Theorem \ref{CompRedTheorem} is proved in \cite{KovRob}. The second half, together with Part 2, is proved in \cite[Theorem 1.2]{derek}.

\subsection{Strategy for the proofs and layout of the paper}\label{Strategy}
Let $G$ be a transitive permutation group [respectively irreducible linear group], of degree [resp. dimension] $n$, and assume that $G$ is imprimitive. Then $G$ may be embedded as a certain subgroup of a wreath product $R\wr S$, where $R$ is a primitive permutation [resp. linear] group of degree [resp. dimension] $r$, $S$ is a transitive permutation group of degree $s$, $rs=n$, and $G\pi=S$, where $\pi:R\wr S\rightarrow S$ denotes projection over the top group. Let $B\cong R^s$ be the base group of $R\wr S$. Then $G\cap B^s$ is ``built" from submodules of induced $G$-modules, and non-abelian $G$-chief factors (see Lemma \ref{prechief}). We will use this to bound the contribution of $G\cap B^s$ to the invariable generator number for $G$. Since $G=(G\cap B^s).(\frac{G}{G\cap B^s})$, we then need to bound $\di(G/G\cap B^s)=\di(S)$. This is done by induction to prove Theorem \ref{TransInvTheorem}. 

For Theorem \ref{PrimInvTheorem}, the affine case of the O'Nan-Scott Theorem is the most difficult to handle, and this requires upper bounds for $\di(G)$ for an irreducible linear group $G$. The approach described in the above paragraph then gives us what we need (here, we bound $\di(G/G\cap B^s)$ by using Theorem \ref{TransInvTheorem}, rather than induction). 

Theorem \ref{CompRedTheorem} Part 1 can easily be reduced to the irreducible case, and we again use the approach described above. For the remaining parts, we use results on the structure of a weakly quasiprimitive linear group from \cite{Asch1}, \cite{Asch2}, \cite{Gor} and \cite{LucMenMor2}.

The layout of the paper is as follows: in Section \ref{IntroSection} we first record some asymptotic results concerning the composition length of finite permutation and linear groups. With the proof strategy outlined above in mind, we will then discuss bounds on the size of a minimal generating set for a submodule of an induced module for a finite group. We close Section \ref{IntroSection} with a discussion of the structure of a weakly quasiprimitive linear group, as mentioned above. In Section \ref{WreathAppSection}, we partially generalise the module theoretic results from the introduction to certain subgroups of wreath products, while Section \ref{InvGenSection} consists of upper bounds for the function $\di$ on various classes of finite groups. Finally, we complete the proof of Theorem \ref{CompRedTheorem} in Section \ref{ProofSection1}, and Theorems \ref{TransInvTheorem} and \ref{PrimInvTheorem} in Section \ref{ProofSection2}.\\
\vspace{5mm}
\noindent {\bf Notation:} We will adopt the $\mathbb{ATLAS}$ \cite{Atlas} notation for group names, although we will usually write $\Sym(n)$ and $\Alt(n)$ for the symmetric and alternating groups of degree $n$. Furthermore, these groups, and their subgroups, act naturally on the set $\{1,\hdots,n\}$; we will make no further mention of this.

The centre of a group $G$ will be written as $Z(G)$, the Frattini subgroup as $\Phi(G)$, and the Fitting subgroup as $F(G)$. The letters $G$, $H$, $K$ and $L$ will usually be used for groups, while $V$ and $W$ will usually be modules. The letter $M$ will usually denote a submodule. 

Throughout, we will use the Vinogradov notation $A\ll B$, which means $A=O(B)$. Finally, ``$\log$" will always mean $\log$ to the base $2$.

\section{Preliminary results}\label{IntroSection}
The purpose of this paper is to study upper bounds for the function $\di$ on certain classes of finite permutation and linear groups. As mentioned in Section 1, the proofs in the most difficult cases essentially amount to using upper bounds on $\di(G)$ for subgroups $G$ of wreath products $R\wr S$. Our main strategy for doing this will be to reduce modulo the base group $B$ of $R\wr S$ and use either induction or previous results to bound $\di(G/G\cap B)$. In this way, all that remains is to investigate the contribution of $G\cap B$ to $\di(G)$.  

As we will show in Lemma \ref{prechief}, the group $G\cap B$ is built, as a normal subgroup of $G$, from submodules of induced modules for $G$, and non-abelian chief factors of $G$. The following lemma shows that, in the abelian case, it therefore suffices to study generator numbers for these submodules as $G$-modules, rather than their invariable generator numbers as groups themselves. More precisely, we have
\begin{Lemma}[{\bf \cite{DetLuc}, Lemma 2}]\label{diMin} Let $G$ be a group and let $N$ be a normal subgroup of $G$. Then\begin{enumerate}[(i)]
\item $\di(G)\le \di(G/N)+\di(N)$.
\item If $N$ is abelian, then $\di(G)\le \di(G/N)+d_G(N)$, where $d_G(N)$ denote the minimal number of generators required to generate $N$ as a $G$-module;
\end{enumerate}\end{Lemma}
Lemma \ref{diMin} also shows that the contribution to $\di(G)$ of the non-abelian building blocks from $G\cap B$ can be bounded above by $2c_{nonab}(R)$, where $c_{nonab}(R)$ denotes the number of non-abelian chief factors of $R$.

With the above in mind, the purpose of this section is two-fold: to investigate the number of ``building blocks" in $G\cap B$ (this will, in most cases, come down to investigating the chief length of the group $R$), and to investigate the contribution of each building block to $\di(G)$. We do this in Sections \ref{CompSection} and \ref{MainModuleSection} respectively. 

\subsection{Composition length and invariable generation in permutation and linear groups}\label{CompSection}
In this section we record numerous results concerning invariable generation and composition length in finite permutation and linear groups. We begin with composition length. 

\begin{Definition} Let $G$ be a group.\begin{enumerate}[(a)]
\item Write $a(G)$ to denote the composition length of $G$.
\item Let $\ab(G)$ and $\anab(G)$ denote the number of abelian and non-abelian composition factors of $G$, respectively. 
\item Let $\nab(G)$ denote the number of non-abelian chief factors of $G$.
\end{enumerate}\end{Definition}

The first result is stated slightly differently to how it is stated in \cite{Pyber}.
\begin{Theorem}[{\bf\cite{Pyber}, Theorem 2.10}]\label{Pyb} Let $R$ be a primitive permutation group of degree $r\geq 2$. Then $a(R)\ll\log{r}$. \end{Theorem}
\begin{Theorem}[{\bf\cite{LucMenMor2}, Proposition 9}]\label{Pyb2} Let $\mathbb{F}$ be a finite field, and let $R\le GL_r(\mathbb{F})$ be completely reducible. Then $a(R)\ll r\log{|\mathbb{F}|}$. \end{Theorem}

We now consider permutation representations of finite simple groups.
\begin{Theorem}[{\bf\cite{Nina}, Lemma 2.6}]\label{OutPermProp} Let $T$ be a non-abelian finite simple group, and suppose that $T$ is contained in $\Sym(n)$, with $n\geq 2$. Then $|\Out(T)|\ll \log{n}$.\end{Theorem}

Finally, we record a result of Cameron, Solomon and Turull concerning the composition length of a finite permutation group. Note that we only give a simplified version of their result here. 
\begin{Theorem}[{\bf\cite{Cam}, Theorem 1}]\label{Cam} Let $G$ be a permutation group of degree $n\ge 2$. Then $a{(G)}\ll n$.\end{Theorem}

We now turn to invariable generation of certain classes of finite groups. We begin with the following general result.
\begin{Theorem}[{\bf \cite{KLS}, Theorem 3.1}]\label{MinTheoremInv} Let $G$ be a finite group, and let $M$ be a minimal normal subgroup of $G$. Then $\di(G)\le \di(G/M)+\delta$, where $\delta:=1$ if $M$ is abelian and $\delta:=2$ if $M$ is nonabelian. In particular, $\di(G)\le \ab{(G)}+2\nab{(G)}\le a(G)$. \end{Theorem} 

Next, we note the theorem of Lucchini and Detomi mentioned in Section 1.
\begin{Theorem}[{\bf \cite{DetLuc}, Theorem 1}]\label{DetLucHalfn} Let $G$ be a subgroup of $\Sym{(n)}$. Then $\di(G)\le \left\lfloor \frac{n}{2}\right\rfloor$, except when $n=3$ and $G\cong \Sym{(3)}$.\end{Theorem}

We also have the bound for simple groups, which was also mentioned previously. 
\begin{Theorem}[{\bf\cite{KLS}, Theorem 5.1} and {\bf\cite{GurMalle}, Theorem 1.3}]\label{SimpleTheorem} Let $T$ be a non-abelian finite simple group. Then $\di{(T)}\le 2$.\end{Theorem}

Since the outer automorphism group of a nonabelian finite simple group is either isomorphic to a subgroup of $\Sym(4)\times C_f$, or is an extension of at most three cyclic groups, the next corollary follows immediately from Theorem \ref{SimpleTheorem}.
\begin{Corollary}\label{KLSSimpleCorOut} Let $T$ be a non-abelian finite simple group, and let $H\le \Out(T)$. Then $\di(H)\le 3$. In particular, if $T\le G\le \Aut(T)$, then $\di(G)\le 5$.\end{Corollary}

\subsection{Generating submodules of induced modules for finite groups}\label{MainModuleSection}
In this section, we record a number of results from \cite[Section 4]{GT} concerning generator numbers in submodules of induced modules. We begin with some terminology.
\begin{Definition} Let $M$ be a group, acted on by another group $G$. A \emph{$G$-subgroup} of $M$ is a subgroup of $M$ which is stabilised by $G$. We say that $M$ is \emph{generated as a $G$-group} by $X\subset M$ and write $M=\langle X\rangle_{G}$ if no proper $G$-subgroup of $M$ contains $X$. We will write $d_{G}(M)$ for the cardinality of the smallest subset $X$ of $M$ satisfying $\langle X\rangle_G=M$. Finally, write $M^{\ast}:=M\backslash\{1\}$.\end{Definition}

Note that this notation is consistent with our use of $d_G(M)$ in Lemma \ref{diMin}.
\begin{Definition} Let $G$ be a group, acting on a set $\Omega$. Write $\chi(G,\Omega)$ for the number of orbits of $G$ on $\Omega$.\end{Definition}

To avoid being cumbersome, we will also introduce some notation which will be retained for the remainder of Section \ref{MainModuleSection}: \begin{itemize}
\item Let $G$ be a finite group.
\item Fix a subgroup $H$ of $G$ of index $s\ge 2$.
\item Let $V$ be a module for $H$ of dimension $a$, over a field $\mathbb{F}$.
\item Let $K:=\core_{G}(H)$.
\item Set $W:=V\uparrow^G_H$ to be the induced module.
\item Denote the set of right cosets of $H$ in $G$ by $\Omega$.\end{itemize}

Next, define the constant $b$ as follows,
\begin{align*} b:=\sqrt\frac{2}{\pi}.\end{align*}
We also have the following definitions.
\begin{Definition}\label{KDef} For a positive integer $s$ with prime factorisation $s=p_{1}^{r_{1}}p_{2}^{r_{2}}\hdots p_{t}^{r_{t}}$, set $\omega{(s)}:=\sum r_{i}$, $\omega_{1}{(s)}:=\sum r_{i}p_{i}$, $K(s):=\omega_{1}{(s)}-\omega(s)=\sum r_i(p_i-1)$ and 
$$\ws(s)=\frac{s}{2^{K(s)}}\binom{K(s)}{\left\lfloor\frac{K(s)}{2}\right\rfloor}.$$ \end{Definition}
\begin{Definition}\label{KDef2} For a positive integer $s$ and a prime $p$, define $s_p$ to be the $p$-part of $s$. Also define $\lpp{(s)}:=\max_{p\text{ prime}} s_p$.\end{Definition}

The first main result deals with the case when $G^{\Omega}$ contains a soluble transitive subgroup.
\begin{Theorem}\label{BKGARTheorem} Suppose that $G^\Omega$ contains a soluble transitive subgroup, and let $M$ be a submodule of $W$. Also, denote by $\chi=\chi{(K,V^\ast)}$ the number of orbits of $K$ on the non-zero elements of $V$. Then
$$d_{G}(M)\le \min\left\{a,\chi\right\}\ws(s)\le \min\left\{a,\chi\right\}\left\lfloor\frac{bs}{\sqrt{\log{s}}}\right\rfloor$$ where $b:=\sqrt{2/\pi}$. Furthermore, if $s=p^{t}$, with $p$ prime, then $$d_{G}(M)\le \min\left\{a,\chi\right\}\left\lfloor \frac{bp^{t}}{\sqrt{t(p-1)}}\right\rfloor.$$
\end{Theorem}

\begin{Remark} If $K$ has infinitely many orbits on the non-zero elements of $V$, then we assume, in Theorem \ref{BKGARTheorem}, and whenever it is used in the remainder of the paper, that $$\min\left\{a,\chi\right\}=a.$$\end{Remark}

We now move on to general finite groups (i.e. those $G$ for which $G^{\Omega}$ does not necessarily contain a soluble transitive subgroup). We retain the notation introduced at the beginning of Section \ref{MainModuleSection}. 

We begin with a definition. Recall the definitions of $\ws(s)$, $s_{p}$, and $\lpp{(s)}$ from Definitions \ref{KDef} and \ref{KDef2}.  
\begin{Definition}\label{EDEF} For a prime $p$, set
$$E(s,p):=\min\left\{{\left\lfloor \frac{bs}{\sqrt{(p-1)\log_{p}{s_{p}}}}\right\rfloor,\frac{s}{\lpp{(s/s_{p})}}}\right\}\text{ and }E_{sol}(s,p):=\min\left\{\ws(s),s_{p}\right\}$$
where we take $\left\lfloor bs/\sqrt{(p-1)\log_{p}{s_{p}}}\right\rfloor$ to be $\infty$ if $s_{p}=1$. \end{Definition}

The following is quickly proved after examining the functions $E_{sol}$ and $E$.
\begin{Proposition} Let $p$ be prime. Then $E_{sol}(s,p)\le E(s,p)$. \end{Proposition}

\begin{Remark}\label{JustAbove} For any finite group $G$ and any $G$-module $M$, $d_G(M)$ is bounded above by $\chi(G,M^\ast)$.\end{Remark}

For the remainder of this section, we will make a further assumption: that the field $\mathbb{F}$ has characteristic $p>0$. The main result for general finite groups reads as follows.
\begin{Theorem} \label{pmodlemma} For a prime $q\neq p$, let $P_{q}$ be a Sylow $q$-subgroup of $G$. Also, let $P'$ be a maximal $p'$-subgroup of $G$. Let $M$ be a submodule of $W$. \begin{enumerate}[(i)]
\item If $G$ is soluble, then $$d_{G}(M)\le \min\left\{a,\chi(P'\cap K,V^{\ast})\right\}s_{p}.$$
\item Let $N$ be a subgroup of $G$ such that $N^{\Omega}$ is soluble, and let $s_{i}$, $1\le i\le t$, be the sizes of the orbits of $N$ on $\Omega$. Then\begin{enumerate}[(a)]
\item We have \begin{align*}d_{G}(M)\le &\min\left\{a,\chi(N\cap P'\cap K,V^{\ast})\vphantom{\frac{1}{2}}\right\}\times\sum_{i=1}^{t}\ws{(s_i)}.\end{align*}
\item If $N$ is soluble, and $P_N'$ is a $p$-complement in $N$, then 
 \begin{align*}d_{G}(M)\le &\min\left\{a,\chi(P_N'\cap K,V^{\ast})\vphantom{\frac{1}{2}}\right\}\times\sum_{i=1}^{t}E_{sol}(s_i,p).\end{align*}\end{enumerate}
\item $d_{G}(M)\le \min\left\{a,\chi(P_{q}\cap K,V^\ast)\right\}s/s_{q}$. 
\item Assume that $s_p>1$. Then $$d_{G}(M)\le \min\left\{a,\chi(K,V^{\ast})\right\}\left\lfloor\frac{bs}{\sqrt{\log{s_p}}}\right\rfloor.$$
\end{enumerate} 
\end{Theorem}

We also record some corollaries from \cite{GT}.
\begin{Corollary}\label{pmodlemmaMainCor} Define $E'$ to be $E_{sol}$ if $G^{\Omega}$ contains a soluble transitive subgroup, and $E':=E$ otherwise. Let $M$ be a submodule of $W$. Then \begin{enumerate}[(i)]
\item $d_G(M)\le aE'(s,p)$.
\item Suppose that for some $q\neq p$, a Sylow $q$-subgroup of $K$ acts transitively on the non-identity elements of $V$. Then
$$d_{G}(M)\le \min\left\{\left\lfloor\frac{bs}{\sqrt{\log{s_p}}}\right\rfloor,\frac{s}{s_q}\}\right\rfloor,$$
where the right hand side above is taken to be $\frac{s}{s_q}$ if $s_p=1$.\end{enumerate}
\end{Corollary}

\begin{Corollary}\label{pq} Let $M$ be a submodule of $W$, and fix $0<\alpha<1$. \begin{enumerate}[(i)]
\item If $s_{p}\ge s^{\alpha}$, then $d_{G}(M)\le aE(s,p)\le a\left\lfloor\frac{bs\sqrt{\frac{1}{\alpha}}}{\sqrt{\log{s}}}\right\rfloor$;
\item If $s_{p}\le s^{\alpha}$, then $d_{G}(M)\le aE(s,p)\le a\left\lfloor \frac{\frac{1}{1-\alpha}s}{c'\log{s}}\right\rfloor$;
\item We have 
$$d_{G}(M)\le aE(s,p)\le
\begin{cases}
\left\lfloor \frac{2as}{c'\log{s}}\right\rfloor, & \text{if }2\le s\le 1260,\\
\left\lfloor \frac{abs\sqrt{2}}{\sqrt{\log{s}}}\right\rfloor, & \text{if }s\ge 1261.
\end{cases}
$$\end{enumerate}
\end{Corollary}

\begin{Corollary}\label{pmodlemmacor2} Let $M$ be a submodule of $W$. If $G$ contains a soluble subgroup $N$, acting transitively on $\Omega$, then \begin{align*}d_{G}(M)\le &\min\left\{a,\chi(P_N'\cap K,V^\ast)\right\}E(s,p)\end{align*}
where $P_N'$ is a $p$-complement in $N$.\end{Corollary}

\subsection{The structure of a weakly quasiprimitive linear group}
We close the preliminary section by recording some a series of lemmas from \cite[Section 2]{derek} concerning the structure of a weakly quasiprimitive linear group (see Definition \ref{ImprimitiveDef}).
\begin{Lemma}[{\bf \cite{derek}, Lemma 2.13}]\label{DefLemma} Let $\mathbb{F}$ be a field, and let $R\le GL_{r}(\mathbb{F})$ be finite, irreducible and weakly quasiprimitive. Then $R$ has a characteristic subgroup $K$ such that $K$ is isomorphic to a subgroup $K_{1}$ of $GL_{r/f}(\mathbb{F}_{1})$, for some divisor $f$ of $r$ and some extension $\mathbb{F}_{1}$ of $\mathbb{F}$, with $[\mathbb{F}_{1}:\mathbb{F}]=f$. All characteristic abelian subgroups of $K_{1}$ are contained in $Z(GL_{r/f}(\mathbb{F}_{1}))$, and $K_{1}$ is weakly quasiprimitive. Moreover, $R/K$ is abelian of order at most $f$, and embeds naturally in $\Gal(\mathbb{F}_{1}:\mathbb{F})$.\end{Lemma}

For a finite, irreducible, weakly quasiprimitive subgroup $R$ of $GL_{r}(\mathbb{F})$, we define $(f(R),K_{1}(R))$ to be any pair consisting of a positive integer $f=f(R)$ and a group $K_{1}=K_{1}(R)$ satisfying the conclusion of Lemma \ref{DefLemma}, with $f$ minimal.

We now record some results concerning the generalised Fitting subgroup of a finite group.
\begin{Lemma}[{\bf \cite{derek}, Lemma 2.14}]\label{FStarLemma} Let $L$ be the generalised fitting subgroup of a finite group $R$. Then $L$ is a central product of $Z(R)$, the noncentral subgroups $O_{q_{i}}(R)$ for a set of primes $q_{i}$, and a collection of normal subgroups $U_{j}$ of $R$. Each $U_{j}$ is a central product of $u_{j}\ge 1$ copies of a quasisimple group $T_{j}$, and conjugation by $R$ permutes these copies transitively. Also, $C_{R}(L)=Z(L)$.\end{Lemma}

\begin{Lemma}[{\bf \cite{derek}, Lemma 2.15}]\label{FStarModule} Let $R\le GL_{r}(\mathbb{F})$ be finite and completely reducible, and let $L$, $q_{i}$ and $U_{j}$ be as in Lemma \ref{FStarLemma}. Assume that $\mathbb{F}$ is a splitting field for each central factor of $L$, and let $C$ be a constituent of the natural $L$-module. Then $C$ decomposes as a tensor product of a one-dimensional module for $Z(R)$, irreducible modules $M_{q_{i}}$ for each $O_{q_{i}}(R)$, and irreducible modules $M_{U_{j}}$ for each $U_{j}$.\end{Lemma}

\begin{Lemma}[{\bf \cite{derek}, Lemma 2.16 or \cite{Luc5}, Lemma 1.7}] \label{FStarO} Let $R$ be finite with cyclic centre $Z$, and assume that all abelian characteristic subgroups of $R$ are contained in $Z$. Then each noncentral $O_{q}(R)$ is the central product of its intersection with $Z$ and an extraspecial $q$-group $E$, of order $q^{1+2m}$ say. If $q$ is odd, then $E$ has exponent $q$. Any nontrivial absolutely irreducible $E$-module has dimension $q^{m}$, and $R/C_{R}(O_{q}(R))\le q^{2m}.Sp_{2m}(q)$. Also, $O_{q}(R/EZ)$ is trivial. Finally, the action of $R/EZ$ on $EZ/Z$ is completely reducible.\end{Lemma}

\begin{Lemma}[{\bf \cite{derek}, Lemma 2.17}]\label{FStarT} Let $R\le GL_{r}(\mathbb{F})$ be finite and completely reducible, and let $L$, $U_{j}$, $u_{j}$, $T_{j}$ and $M_{U_{j}}$ be as in Lemmas \ref{FStarLemma} and \ref{FStarModule}. Assume that $\mathbb{F}$ is a splitting field for all central factors of $L$, and that $L$ acts homogeneously. Then $M_{U_{j}}$ is a tensor product of $u_{j}$ copies of some faithful irreducible $\mathbb{F}[T_{j}]$-module $M_{T_{j}}$, of dimension $t_{j}\ge 2$. Also, writing bars to denote reduction modulo $C_{R}(U_{j})$, $\overline{T_{j}}$ is a nonabelian simple group, and $\overline{T_{j}}^{u_{j}}\le \overline{R}\le (\overline{T_{j}}.A_{j})\wr \Sym(u_{j})$, where $A_{j}$ is the subgroup of $\Out({\overline{T_{j}}})$ that stabilises the module $M_{T_{j}}$.\end{Lemma}

\begin{Corollary}\label{FStarTCors2} Let $R\le GL_{r}(\mathbb{F})$ be finite and completely reducible, and let $U_{j}$, $u_{j}$, $A_{j}$, $T_{j}$ and $t_{j}$ be as in Lemma \ref{FStarT}. Write bars to denote reduction modulo $C_{R}(T_{j})$.\begin{enumerate}
\item If $\overline{T_{j}}\not\cong P\Omega^{+}_{8}(q)$, with $q$ odd, then $A_{j}\le N.M$, where $|N|\le 2$ and $M$ is metacyclic. If $\overline{T_{j}}\cong P\Omega^{+}_{8}(q)$, with $q$ odd, then $A_{j}\le N.M$, where $N$ is cyclic and $|M|$ divides $24$;
\item Assume that $t_{j}=2$ for some $j$. Then $\overline{R}\le \overline{T_{j}}^{u_{j}}.(C_{2}\wr\Sym{(u_{j})})$.\end{enumerate}\end{Corollary} 
\begin{proof} If $t_{j}=2$, then \cite[Lemma 2.10]{derek} implies that $|A_{j}|\le 2$, so (ii) immediately follows. If $\overline{T_{j}}$ is not $P\Omega^{+}_{8}(q)$, with $q$ odd, then part (i) follows from examining the structure of $\Out{(\overline{T_{j}})}$ (see \cite[Chapter 2]{KleidLie}). If $\overline{T_{j}}=P\Omega^{+}_{8}(q)$, with $q$ odd, then 
$\Out{(\overline{T_{j}})}$ modulo its cyclic group of field automorphisms, has order $24$,  by \cite[Theorem 2.1.4 and Proposition 2.7.3]{KleidLie}.\end{proof}

\begin{Corollary}\label{FStarCor} Let $F$ be a finite field, let $R\le GL_{r}(\mathbb{F})$ be irreducible and weakly quasiprimitive, let $f=f(R)$, $K_{1}=K_{1}(R)$, let $L$ be the generalised fitting subgroup of $K_{1}$, and let $O_{q_{i}}(K_{1})$, $U_{j}$ be the central factors of $L$ as in Lemma \ref{FStarLemma}. Also, for each $q_{i}$ let $m_{i}$ be as in Lemma \ref{FStarO}, and for each $U_{j}$, let $u_{j}$ and $t_{j}$ be as in Lemma \ref{FStarT}. Then\begin{enumerate}[(i)]
\item $\prod_{i} q_{i}^{m_{i}}\prod_{j} t_{j}^{u_{j}}$ divides $r/f$;
\item Each $q_{i}$ divides $|\mathbb{F}|^{f}-1$;
\end{enumerate}\end{Corollary}
\begin{proof} By definition, each characteristic abelian subgroup of $K_{1}$ is contained in $Z:=Z(GL_{r}(\mathbb{F}^{f}))$. Since $q_{i}$ divides $|Z(O_{q_{i}}(K_{1}))|$ and $|Z|=|\mathbb{F}|^{f}-1$, part (ii) follows.

Next, extend $\mathbb{F}_{1}:=\mathbb{F}_{|F|^{f}}$ so that $\mathbb{F}_{1}$ is a splitting field for all subgroups of $L$, and let $M_{q_{i}}$ and $M_{U_{j}}$ be as in Lemma \ref{FStarModule}. Then $K_{1}$ may no longer be weakly quasiprimitive. In particular, $L$ may no longer be homogeneous, but its irreducible constituents are algebraic conjugates of one another, so they all have the same dimension $e$. Hence, $e$ divides $r/f$. Since Lemmas \ref{FStarModule}, \ref{FStarO} and \ref{FStarT} imply that $e=\prod_{i} \dim{M_{q_{i}}}\prod_{j} \dim{M_{U_{j}}}=\prod_{i} q_{i}^{m_{i}}\prod_{j} t_{j}^{u_{j}}$, part (i) follows.
\end{proof}

The following is quickly computed using the database of irreducible matrix groups in MAGMA.
\begin{Proposition}\label{83Cor} Let $R\le GL_{r}(p)$ be primitive. If $(r,p)=(4,2),(4,3),(6,2),(6,3)$ or $(8,2)$, then the list of composition factors of $R$ is contained in one of the lists in Table 1 below.\end{Proposition}
\begin{tabular}[t]{|c|p{3.5cm}|p{10.5cm}|}
  \hline
  \multicolumn{3}{ |c| }{Table 1}\\ 
  \hline
  $(r,p)$ & Maximum value of $a(R)$ among the primitive irreducible subgroups of $GL_{r}(p)$ & List of composition factors of a primitive irreducible subgroup $R$ of $GL_{r}(p)$ is contained in one of the lists below\\
  \hline\hline
  $(4,2)$ & $4$ & $[2,2,3,5]$ ($f(R)=4$), $[2,A_{5}]$, $[2,A_{6}]$, $[A_{7}]$, $[L_{4}(2)]$\\
  \hline
  $(4,3)$ & $10$ & $[2,2,2,2,2,2,2,2,3,3]$, $[2,2,2,2,2,2,2,5]$, $[2,2,2,2,2,A_{6}]$, $[2,2,2,2,2,2,A_{5}]$, $[2,2,U_{4}(2)]$, $[2,2,L_{4}(3)]$\\
  \hline
  $(6,2)$ & $8$ & $[2,3,3,3,7]$, $[2,2,2,2,3,3,3,3]$, $[2,3,3,S]$, $[3,7,S]$, where $S=L_{2}(7)$, $L_{2}(8)$, $A_{6}$, $A_{7}$, $U_{3}(3)$, $A_{8}$, $L_{3}(4)$, $U_{4}(2)$, $PSp_{6}(2)$ or $L_{6}(2)$ \\
  \hline
  $(6,3)$ & $7$ & $[2,2,2,2,3,7,13]$, $[2,2,2,2,3,3,13]$, $[2,2,2,2,A_{5}]$, $[2,2,2,2,L_{2}(7)]$, $[2,2,2,2,2,A_{6}]$, $[2,L_{2}(p)]$ ($p=11$, $13$), $[2,2,A_{7}]$, $[2,2,2,2,U_{3}(3)]$, $[2,2,2,2,3,L_{3}(3)]$, $[2,2,3,13,L_{2}(27)]$, $[2,2,2,L_{3}(4)]$, $[2,M_{12}]$, $[2,2,2,L_{4}(3)]$, $[2,2,2,2,U_{4}(3)]$, $[2,2,2,2,L_{3}(9)]$, $[2,2,PSp_{6}(3)]$, $[2,2,L_{6}(3)]$ \\
  \hline
  $(8,2)$ & $7$ & $[2,2,2,3,5,17]$ ($f(R)=8$), $[2,2,3,3,5,S]$, $[2,2,3,S,S]$, $[2,3,3,S,S]$, where $S=A_{5}$, $L_{2}(7)$, $A_{6}$, $L_{2}(8)$, $L_{2}(17)$, $L_{2}(16)$, $A_{7}$, $U_{4}(2)$, $A_{8}$, $A_{9}$, $A_{10}$, $PSp_{4}(4)$, $PSp_{6}(2)$, $P\Omega^{+}_{8}(2)$, $P\Omega^{-}_{8}(4)$, $L_{4}(4)$, $PSp_{8}(2)$, $L_{8}(2)$\\
  \hline \end{tabular}

\section{Wreath products}\label{WreathAppSection}
\subsection{Wreath products as permutation and linear groups}
Let $R$ be a finite group, let $S$ be a permutation group of degree $s$, and consider the wreath product $R\wr S$, as constructed in \cite{CamPerm}. Let $B$ be the base group of $R\wr S$, so that $B$ is isomorphic to the direct product of $s$ copies of $R$. Thus, for a subgroup $L$ of $R$, $B$ contains the direct product of $s$ copies of $L$: we will denote this direct product by $B_{L}$ (so that $B_{1}=1$ and $B_{R}=B$). 

Now, for each $1\le i\le s$, set 
$$R_{(i)}:=\{(g_1,\hdots,g_s)\in B\text{ : }g_j=1\text{ for all }j\neq i\}\unlhd B.$$ Then $R_{(i)}\cong R$, and $B=\prod_{1\le i\le s} R_{(i)}$. Furthermore, $N_{R\wr S}(R_{(i)})\cong R_{(i)}\times (R\wr \Stab_{S}(i))$. Hence, we may define the projection maps 
\begin{align}\rho_{i}:N_{R\wr S}(R_{(\gamma)})\rightarrow R_{(i)}.\end{align} We also define $\pi:R\wr S\rightarrow S$ to be the quotient map by $B$. This allows us to define a special class of subgroups of $R\wr S$.
\begin{Definition}\label{LargeDef} A subgroup $G$ of $R\wr S$ is called \emph{large} if \begin{enumerate}[(a)]
\item $G\rho_{i}=R_{(i)}$ for all $i$ in $1\le i\le s$, and;
\item $G\pi=S$.\end{enumerate}\end{Definition}

\begin{Remark}\label{WreathBlockRemark} Suppose, in addition, that $R$ is a finite permutation [respectively irreducible linear] group of degree [resp. dimension] $r\geq 1$ [we exclude $r=1$ in the permutation group case]. If $s>1$ and $G$ is a large subgroup of $R\wr S$, then $G$ is an imprimitive permutation [resp. linear] group of degree $rs$, with a system of $s$ blocks, each of cardinality [resp. dimension $r$]. ($G$ acts on the cartesian product $\{1,\hdots,r\}\times \{1,\hdots,s\}$ in the permutation group case.) \end{Remark}
In fact, it turns out that all imprimitive permutation [resp. linear] groups arise as a large subgroup of a certain wreath product. 

\begin{Theorem}[{\bf \cite{Sup}, Theorem 3.3}]\label{SupPerm} Let $G$ be an imprimitive permutation group on a set $\Omega_1$, and let $\Delta$ be a block for $G$. Also, let $\Gamma:=\Delta^{G}$ be the set of $G$-translates of $\Delta$, and set $\Omega_2:=\Delta\times \Gamma$. Denote by $R$ and $S$ the permutation groups $\Stab_{G}(\Delta)^{\Delta}$, and $G^{\Delta^{G}}$, on $\Delta$ and $\Gamma$ respectively. Then\begin{enumerate}[(i)]
\item $G\cong G^{\Omega_{2}}$ is isomorphic to a large subgroup of $R\wr S$, and;
\item $(G,\Omega_{1})$ and $(G,\Omega_{2})$ are permutation isomorphic.\end{enumerate}\end{Theorem}

If $G$ is an imprimitive permutation group, and the block $\Delta$ as in Theorem \ref{SupPerm} is assumed to be a minimal block for $G$, then the group $R=\Stab_{G}(\Delta)^{\Delta}$ is primitive. When $\Omega$ is finite we can iterate this process, and deduce the following.
\begin{Corollary}\label{PrimCom} Let $G$ be a transitive permutation group on a finite set $\Omega$. Then there exist primitive permutation groups $R_{1}$, $R_{2}$, $\hdots$, $R_{t}$ such that $G$ is a subgroup of $R_{1}\wr R_{2}\wr \hdots\wr R_{t}$.\end{Corollary}

\begin{Remark}\label{WreathRemark} The wreath product construction is associative, in the sense that $R\wr (S\wr T)\cong (R\wr S)\wr T$, so the iterated wreath product in Corollary \ref{PrimCom} is well-defined.\end{Remark}

\subsection{An application of the results in Section \ref{MainModuleSection} to wreath products}
We first make the following easy observation.
\begin{Proposition}\label{ISD} Let $A=T_{1}\times T_{2}\times\hdots\times T_{f}$, where each $T_{i}$ is isomorphic to the nonabelian finite simple group $T$. Suppose that $M\le A$ is a subdirect product of $A$, and suppose that $M'\unlhd M$ is also a subdirect product of $A$. Then $M'=M$.\end{Proposition}
\begin{proof} We prove the claim by induction on $f$, and the case $f=1$ is trivial, so assume that $f>1$. Since $M$ is subdirect, each $M\cap T_{i}$ is normal in $T_{i}$. If $M=A$, then since the only normal subgroups of $A$ are the groups $\prod_{i\in Y} T_{i}$, for $Y\subseteq \{1,\hdots,f\}$, the result is clear. So assume that $M\cap T_{i}=1$ for some $i$. Then $M'\cap T_{i}=1$, and $M'T_{i}/T_{i}$ and $MT_{i}/T_{i}$ are subdirect products of $\prod_{j\neq i} T_{j}$. It follows, using the inductive hypothesis, that $M'T_{i}=MT_{i}$. Hence $M'=M$, since $M\cap T_{i}=1$, and the proof is complete. \end{proof}  

We will now fix some notation which will be retained for the remainder of the section.\begin{itemize}
\item Let $R$ be a finite group (we do not exclude the case $R=1$).
\item Let $S$ be a transitive permutation group of degree $s\ge 2$.
\item Let $G$ be a large subgroup of the wreath product $R\wr S$ (see Definition \ref{LargeDef}).
\item Write $B:=R_{(1)}\times R_{(2)}\times\hdots\times R_{(s)}$ for the base group of $R\wr S$.
\item write $\pi:G\rightarrow S$ for the projection homomorphism onto the top group.
\item Let $H:=N_{G}(R_{(1)})=\pi^{-1}(\Stab_S(1))$.
\item Let $\Omega:=H\backslash G$.
\item Let $K:=G\cap B=\core_G(H)=\Ker_G(\Omega)$. \end{itemize}
Recall that for a subgroup $N$ of $R$, $B_N\cong N^s$ denotes the direct product of the distinct $S$-conjugates of $N$. In particular, if $N\unlhd R$, then $B_{N}\unlhd R\wr S$. Throughout, we will view $R$ as a subgroup of $B$ by identifying $R$ with $R_{(1)}$. We also note that 
\begin{itemize}
\item $|G:H|=s$; and
\item $S=G^{\Omega}$.\end{itemize}
In particular, the notation is consistent with the notation introduced at the beginning of Section \ref{MainModuleSection}.
 
\begin{Remark}\label{NEWREM} If $R$ is a transitive permutation group, acting on a set $\Delta$, then $G$ is an imprimitive permutation group acting on the set $\Delta\times \{1,2,\hdots,s\}$, and $H=\Stab_G((\Delta,1))$. Furthermore $H^{\Delta}=R$, since $G$ is large (see Remark \ref{WreathRemark}).\end{Remark}

Our strategy for proving Theorems \ref{TransInvTheorem} and \ref{CompRedTheorem} can now be summarised as follows:\begin{description}
\item[Step 1:] Show that $K$ is ``built" from induced modules for $G$, and non-abelian $G$-chief factors.
\item[Step 2:] Derive bounds on $\di(G)$ in terms of the factors from Step 1 and $\di(S)$.
\item[Step 3:] Use the results from Section \ref{MainModuleSection}, to bound the contributions from the factors in Step 1 to the bound from Step 2.
\item[Step 4:] Use induction/previous results to bound $\di(S)$.\end{description}

We begin with Step 1. 
\begin{Lemma}\label{prechief} Suppose that $R>1$ and that $1:=N_0\le N_1\le\hdots\le N_e=R$ is a normal series for $R$, where each factor is either elementary abelian, or a nonabelian chief factor of $R$. Consider the corresponding normal series $1:=G\cap B_{N_0}\le G\cap B_{N_1}\le\hdots\le G\cap B_{N_e}=G$ for $G$. Let $V_i:={N_{i}}/{N_{i-1}}$ and $M_i:=G\cap B_{N_i}/G\cap B_{N_{i-1}}$. \begin{enumerate}[(i)]
\item If $V_i$ is elementary abelian, then $M_i$ is a submodule of the induced module $V_i\uparrow^G_H$.
\item If $V_i$ is a nonabelian chief factor of $R$, then $M_i$ is either trivial, or a nonabelian chief factor of $G$.
\end{enumerate}\end{Lemma}

For the remainder of this section, suppose that $1:=N_0\le N_1\le\hdots\le N_e=R$ is a chief series for $R$, and let $V_i:={N_{i}}/{N_{i-1}}$ and $M_i:=G\cap B_{N_i}/G\cap B_{N_{i-1}}$. If $V_i$ is abelian we will also write $|V_i|=p_{i}^{a_{i}}$, for $p_i$ prime. 

We now have Step 2. 
\begin{Corollary}\label{chiefpreCor} We have
$$\di(G)\le \sum_{V_i\text{ abelian}} d_G(M_i)+2\nab{(R)}+\di(S)$$
\end{Corollary}
\begin{proof} We will prove the corollary by induction on $|R|$. If $|R|=1$ then the bound is trivial, since $G\cong S$ in that case, so assume that $|R|>1$, and note that 
\begin{align}\label{ASTR} G/M_1\text{ is a large subgroup of }(R/V_1)\wr S.\end{align}
Suppose first that $V_1$ is abelian. Then $M_1$ is a $G$-module, so
\begin{align*} d(G)\le d_{G}(M_1)+ \di(G/M_1).\end{align*}
by Lemma \ref{diMin}. Since $\nab{(R)}=\nab{(R/V_1)}$, (\ref{ASTR}) and the inductive hypothesis give the result.

So we may assume that $V_1$ is nonabelian. Then $M_1$ is either trivial or a minimal normal subgroup of $G$, by Lemma \ref{prechief} Part (ii). Hence, $\di(G)\le \di(G/M_1)+2$ by Lemma \ref{diMin}. The result now follows, again from (\ref{ASTR}) and the inductive hypothesis.\end{proof}     

Before stating our next corollary, we refer the reader to Definition \ref{EDEF} for a reminder of the definitions of the functions $E$ and $E_{sol}$. The next two corollaries deal with Step 3.
\begin{Corollary}\label{chief} Define $E'$ to be $E_{sol}$ if $S$ contains a soluble transitive subgroup, and $E':=E$ otherwise. Then\begin{enumerate}[(i)]
\item $\di(G)\le \sum_{V_i\text{ abelian}} a_iE'(s,p_{i})+2c_{nonab}{(R)}+d(S)$.
\item Suppose that $|R|=2$ and $s=2^{m}q$, where $q$ is odd, and that $S$ has a tuple of primitive components $X=(R_2,\hdots,R_t)$, where $\bl_{X,2}(S)\ge 1$. Let $\Gamma$ be a full set of blocks for $S$ of size $2^{\bl_{X,2}(S)}$, and set $\widetilde{S}:=S^\Gamma$. Then $$\di(G)\le \sum_{i=0}^{\bl_{X,2}(S)} E'(2^{m-i}q,2)+\di(\widetilde{S}).$$\end{enumerate}
\end{Corollary}
\begin{proof} By Corollary \ref{chiefpreCor}, we have
$$\di(G)\le \sum_{V_i\text{ abelian}} d_G(M_i)+2\nab{(R)}+\di(S).$$
Now, by Corollary \ref{pmodlemmaMainCor}, $d_G(M_i)\le a_iE'(s,p_i)$. This proves (i). 

So we consider Part (ii). We will show that 
\begin{align}\label{MAD} \di(S)\le\sum_{i=1}^{\bl_{X,2}(S)} E(2^{m-i}q,2)+\di(\widetilde{S})\end{align}
by induction on $\bl_{X,2}(S)$. The result will then follow, since $d(G)\le E'(2^mq,2)+d(S)$ by Part (i). Now, by hypothesis, $S$ has a tuple of primitive components $X=(R_2,\hdots,R_t)$. Also, $|R_2|=2$ since $\bl_{X,2}(S)\ge 1$. Hence, by Theorem \ref{SupPerm}, $S$ is a large subgroup of a wreath product $R_2\wr S_2$, where either $S_2=1$, or $S_2$ is a transitive permutation group of degree $2^{m-1}q$, with a tuple $Y:=(R_3,\hdots,R_t)$ of primitive components. If $S_2=1$ then the result follows, since $s=4$ and $\widetilde{S}=1$ in that case. So assume that $S_2>1$. By Part (i), we have 
\begin{align}\label{Man} \di(S)\le E'(2^{m-1}q,2)+\di(S_2)\end{align}
If $\bl_{X,2}(S)=1$ then $S_2=\widetilde{S}$ and (\ref{MAD}) follows from (\ref{Man}). So assume that $\bl_{X,2}(S)>1$. Then $\bl_{Y,2}(S_2)=\bl_{X,2}(S)-1\ge 1$. The inductive hypothesis then yields $\di(S_2)\le \sum_{i=1}^{\bl_{Y,2}(S_2)} E(2^{m-1-i}q,2)+\di(\widetilde{S})= \sum_{i=2}^{\bl_{X,2}(S)} E(2^{m-i}q,2)+\di(\widetilde{S})$. The bound (\ref{MAD}) now follows immediately from (\ref{Man}), which completes the proof.
\end{proof}

The following is immediate from  Corollaries \ref{pq} and \ref{chief} Part (i).
\begin{Corollary}\label{37} Suppose that $|R|\geq 2$. We have
\begin{align*}\di(G) &\ll \frac{\ab(R)s}{\sqrt{\log{s}}}+2\nab{(R)}+\di(S)\le \frac{a(R)s}{\sqrt{\log{s}}}+\di(S). \end{align*}
\end{Corollary}

The next corollary will be key in our proof of Theorem \ref{TransInvTheorem} when $G$ is imprimitive with minimal block size $4$. 
\begin{Corollary}\label{S4NewBoundCor} Assume that $R=S_{4}$ or $R=A_{4}$. Define $E'$ to be $E_{sol}$ if $S$ contains a soluble transitive subgroup, and $E':=E$ otherwise. Then 
$$\di(G)\le E'(s,2)+\min\left\{\frac{bs}{\sqrt{\log{s_2}}},\frac{s}{s_3}\right\} +E'(s,3)+\di(S).$$
\end{Corollary}
\begin{proof} We have
\begin{align}\label{hat}\di(G)\le d_G(M_1)+d_G(M_2)+d_G(M_3)+\di(S)\le 2E'(s,2)+E'(s,3)+d_G(M_3)+\di(S) \end{align}
by Corollaries \ref{chiefpreCor} and \ref{pmodlemmaMainCor} Part (i). Let $\Delta:=\{1,2,3,4\}$, so that $R$ is transitive on $\Delta$. We have $V_1\cong 2^{2}$, $V_2\cong 3$, and $V_3\cong 2$ if $R\cong S_4$. Since $K^{\Delta}$ is a normal subgroup of $H^{\Delta}=R$ (see Remark \ref{NEWREM}), $K^{\Delta}$ is isomorphic to either $2^{2}$, $A_4$, or $S_4$. In the first two cases $M_{3}$ is trivial. Hence, since 
$$E'(s,2)\le \min\left\{\frac{bs}{\sqrt{\log{s_2}}},\frac{s}{s_3}\right\},$$the result follows from (\ref{hat}).

So we may assume that $K^{\Delta}\cong S_4$. Then a Sylow $3$-subgroup $P_3$ of $K^{\Delta}$ acts transitively on the non-identity elements of $V_1$. Thus, $\chi(P_{3}\cap K,V_1^\ast)=1$, so
$$d_G(M_1)\le \min\left\{\frac{bs}{\sqrt{\log{s_2}}},\frac{s}{s_3}\right\}$$
by Corollary \ref{pmodlemmaMainCor} Part (ii), with $(p,q):=(2,3)$. The result now follows after applying Corollary \ref{pmodlemmaMainCor} Part (i) to $d_G(M_2)$ and $d_G(M_3)$.\end{proof}

\begin{Corollary}\label{6CaseCor} Let $s\ge 2$, let $S\le \Sym{(s)}$ be transitive, and let $R\le GL_{6}(2)$ be primitive and irreducible, with the following properties:\begin{enumerate}
\item $R$ has a characteristic subgroup $K_{1}$, such that $Z:=Z(K_{1})$ has order dividing $3$, and $\overline{K_{1}}:=K_{1}/Z$ has shape shape $N.X$, with $N\unlhd \overline{K_{1}}$ elementary abelian of order $3^{2}$, and $X\le Sp_{2}(3)$ completely reducible;
\item $|R/K_{1}|\le 2$.\end{enumerate}
Let $G$ be a large subgroup of the wreath product $R\wr S$, and let $Y\unlhd X$ be the induced action of $(G\cap K_{1}^{s})/(G\cap Z^{s})$ on $N$. Then $\di(G)\le E(s,2)+E(s,2)+\min\left\{\lfloor bs/\sqrt{\log{s_{2}}}\rfloor,s/s_{3}\right\}+E(s,3)+\min\left\{\lfloor bs/\sqrt{\log{s_{3}}}\rfloor,s/s_{2}\right\}+E(s,3)+\di(S)$, unless $|Y|=8$, in which case $\di(G)\le 4E(s,2)+\min\left\{\lfloor bs/\sqrt{\log{s_{3}}}\rfloor,s/s_{2}\right\}+E(s,3)+\di(S)$. Furthermore, if $s=2$ then $\di(G)\le 6$.\end{Corollary}
\begin{proof} Since $Y\le Sp_{2}(3)$ is completely reducible, $|Y|$ must be $1$, $2$, $4$, $8$ or $24$. Suppose first that $s\neq 2$. If $|Y|$ is $1$, $2$ or $4$, then $R$ is soluble of order dividing $2^{3}3^{3}$, and the result follows from Corollary \ref{chief} Part (i), since $E(s,2)\le \min\left\{\lfloor bs/\sqrt{\log{s_{2}}}\rfloor,s/s_{3}\right\}$, and $E(s,3)\le \min\left\{\lfloor bs/\sqrt{\log{s_{3}}}\rfloor,s/s_{2}\right\}$. If $|Y|=8$, then $Y$ acts transitively on the nonidentity elements of $N$, and the result follows from Corollary \ref{pmodlemmaMainCor} Part (ii) with $(p,q)=(3,2)$, and Corollary \ref{chiefpreCor}. If $|Y|=24$, then a Sylow $2$-subgroup of $Y$ acts transitively on the nonidentity elements of $N$; furthermore, a Sylow $3$-subgroup of $Y/Z(Y)\cong A_{4}$ acts transitively on the nonidentity elements of the Klein $4$-subgroup of $Y/Z(Y)$. The result then follows, again from Corollary \ref{pmodlemmaMainCor} Part (ii) and Corollary \ref{chiefpreCor}.

Finally, assume that $s=2$. We need to prove that $G$ can be invariably generated by $6$ elements. Let $M\le K_{1}$ such that $M/Z\cong N$. Also, let $H\le K_{1}$ such that $M\le H$ and $H/M=(K_{1}\cap Z(Sp_{2}(3)))/M\le C_{2}$. Arguing as in the paragraph above, we have $\di(G/G\cap H^{2})\le E(2,2)+2/2_{2}+E(2,3)+\di(S)=4$ if $|Y|= 24$; $\di(G/G\cap H^{2})\le E(2,2)+2E(2,2)+\di(S)=4$ if $|Y|=8$; and $\di(G/G\cap H^{2})\le E(2,2)+E(2,2)+\di(S)=3$ if $|Y|<8$. Thus, we just need to show that $A:=G\cap H^{2}$ can be generated, as a $G$-group, by $2$ elements if $|Y|\ge 8$, and $3$ elements if $|Y|< 8$. Now, by Corollary \ref{pmodlemmaMainCor} Part (i), $G\cap Z^{2}$ can be generated, as a $G$-module, by $E(2,3)=1$ element; let $x$ be such an element, so that $|x|=3$. Also, $H/Z$ has shape $N.L$, where $L\le Z(Sp_{2}(3))$. Suppose first that $|Y|\geq 8$. Since a Sylow $2$-subgroup of $Sp_{2}(3)$ acts transitively on the nonidentity elements of $N$, Corollary \ref{pmodlemmaMainCor} Part (ii) implies that $G\cap M^{2}/G\cap Z^{2}$ can be generated, as a $G$-module, by $2/2_2=1$ element. Say $x_{1}\in G\cap M^{2}$ is the preimage of such an element. 

Next, suppose that $|Y|< 8$. Then again using Corollary \ref{pmodlemmaMainCor} Part (ii) $G\cap M^{2}/G\cap Z^{2}$ can be  generated, as a $G$-module, by $2E(s,3)=2$ elements. Say $x_{1}$, $x_{2}\in G\cap M^{2}$ are preimages. 

Now, using Corollary \ref{pmodlemmaMainCor} Part (i), $A/G\cap M^{2}$ can be generated, as a $G$-module, by $E(2,2)=1$ element. Say $y\in A$ is the preimage of such an element. Clearly we may assume that $|y|$ is a power of $2$. Thus, since $y\in C_{K_{1}}(x)$, it is now easy to see that $\left\{x_{1},xy\right\}$ (or $\left\{x_{1},x_{2},xy\right\}$ if $|Y|<8$) generates $A$ as a $G$-group, and this completes the proof.\end{proof}                

\section{Minimal invariable generator numbers in certain classes of finite groups}\label{InvGenSection}
In this section, we consider upper bounds for the function $\di$ on various classes of finite groups. We begin with bounds on $\di(G)$ for some transitive permutation groups $G$.
\begin{Proposition}\label{TransProp} Let $G$ be a transitive permutation group of degree $n$.\begin{enumerate}[(i)]
\item If $n=6$, then $\di(G)\le 2$, except that $\di(S_{6})=3$;
\item If $n=8$, then $\di(G)\le 4$, and $\di(G)=4$ if and only if $G\cong D_{8}\circ D_{8}$;
\item If $n=9$, then $\di(G)\le 3$;
\item If $n=10$, then $\di(G)\le 3$;
\item If $n=12$, then $\di(G)\le 4$;
\item If $n=16$, then $\di(G)\le 6$;
\item If $n=18$, then $\di(G)\le 4$.
\end{enumerate}\end{Proposition}
\begin{proof} By \cite[Lemma 2.1]{KLS}, a subset $X$ of $G$ invariably generates $G$ if and only if th following holds: for each maximal subgroup $M$ of $G$, at least one element of $X$ acts fixed point freely on the $G$-cosets of $M$. Using this, and the database of transitive permutation groups of small degree (see \cite{CanHol}), one can readily check that the result holds in each of the listed cases, using MAGMA.\end{proof}

Next, we study subgroups of wreath products in which the bottom group is cyclic.
\begin{Proposition}\label{Cmwr3} Let $G$ be a large subgroup in the wreath product $C_{m}\wr S_{3}$, where $C_{m}$ denotes the cyclic group of order $m$. Then $\di(G)\le 4$.\end{Proposition}
\begin{proof}Let $B$ be the base group of $W:=C_{m}\wr S_{3}$, so that $\di(G)\le d_{G}(G\cap B)+\di(G/G\cap B)$. Since $d_{G}(G\cap B)\le d(G\cap B)\le 3$, we may assume that $G/G\cap B\cong S_{3}$.

Now, writing $B$ in additive notation, we next set $B_{1}:=\left\{(a,b,c)\in B\cap B\text{ : }a+b+c=0\right\}\le B$. Then $B_{1}$ is a $W$-submodule of $B$. Furthermore, $W_{1}\cong C_{m}^{2}$, so $N_{1}:=G\cap W_{1}$ is metacyclic. Hence, since Lemma \ref{diMin} gives $\di(G)\le d_{G}(N_{1})+\di(G/N_{1})$, it will suffice to prove that $d_{i}(G/N_{1})\le 2$. 

Writing bars to denote reduction modulo $N_{1}$, we have $\overline{G}\cong (\overline{G\cap B}).S_{3}$. Furthermore, it is easily seen that $\overline{G\cap B}$ is cyclic, and is contained in $Z(G)$; write $\overline{G\cap B}=3^{k}\times t$, where $(3,t)=1$, and let $x$ and $y$ be elements of $\overline{G\cap B}$ of order $3^{k}$ and $t$, respectively. Also, let $a$ and $b$ be elements of $\overline{G}$ of $2$-power and $3$-power order, respectively, which reduce to a $2$-cycle and a $3$-cycle modulo $\overline{G\cap B}$. Then clearly the set $\left\{xa,yb\right\}$ invariably generates $G$, and the result follows.\end{proof}

The next two results extend Lemmas 3.3 and 3.4 in \cite{derek}; indeed, our proofs use the same techniques as used therein.
\begin{Proposition}\label{2Prop} Let $G$ be a subgroup in the wreath product $C_{2}\wr \Sym{(u)}$. Then $\di(G)\le u$, and if $\di(G)=u$, then $G$ is a $2$-group and $u$ is even.\end{Proposition}
\begin{proof} We prove the claim by induction on $u$, and the case $u=1$ is clear, so assume that $u>1$. Since $G\le \Sym{(2u)}$, the claim $\di(G)\le u$ follows immediately from Theorem \ref{DetLucHalfn}. So assume that $\di(G)=u$. We need to show that $G$ is a $2$-group and that $u$ is even. Let $\pi:G\rightarrow \Sym{(u)}$ denote the projection over the top group. If $\pi(G)$ is intransitive, with an orbit $\Delta$ of size $v$, then the induced action of $G$ on the preimage of $\Delta$ in $\left\{1,\hdots,2u\right\}$ is contained in $C_{2}\wr \Sym{(v)}$, with kernel contained in $C_{2}\wr \Sym{(u-v)}$; hence, the result follows from the inductive hypothesis.

So assume that $\pi(G)$ is transitive. If $\ker(\pi)$ is trivial, then $G\le \Sym{(u)}$, so $\di(G)\le (u+\delta_{3,u})/2<u$. Thus, $G$ must be a transitive subgroup of $C_{2}\wr \Sym{(u)}$. If $u=3$ then $\di(G)\le 2$ by direct computation, so we must have $u\neq 3$. Using Corollary \ref{chief} and Theorem \ref{DetLucHalfn}, we have $u=\di(G)\le E(u,2)+u/2$. Using the definition of the function $E$, one easily sees that $E(u,2)<u/2$ unless $u=2$ or $u=4$. If $u=2$ then $G\le D_{8}$. If $u=4$, then $\di(G)=4$ implies that $G\cong D_{8}\circ D_{8}$ by Proposition \ref{TransProp}. The result follows.\end{proof}

\begin{Proposition}\label{S3Prop} Let $G\le S_{3}^{u}$. Then $\di(G)\le u$, except that $\di(G)=u+1$ when $G\cong 3^{u}:2$, with $Z(G)=1$.\end{Proposition}
\begin{proof} We prove the claim by induction on $u$, and the case $u=1$ is clear, so assume that $u>1$. Let $K$ be the kernel of the projection of $G$ onto the first $u-1$ direct factors of $S_{3}^{u}:=B_{1}\times B_{2}\times \hdots\times B_{u}$. 

Suppose first that $G/K\cong 3^{u-1}:2$, with $Z(G/K)=1$. Then it is easy to see that $G/K$ is invariably generated by a set $\left\{Kb_{1},Kb_{2},\hdots,Kb_{u}\right\}$, where $|Kb_{i}|=3$ for $1\le i\le u-1$, and $|Kb_{u}|=2$. Clearly we can assume that $|K|>1$. If $K=\langle x\rangle$ with $|x|=2$, then $|xb_{1}|=6$ and $\left\{xb_{1},b_{2},\hdots,b_{u}\right\}$ invariably generates $G$. If $|K|=6$, then we may assume that the projection of $b_{1}$ onto $B_{1}$ is a $2$-cycle, and that the projection of $b_{u}$ onto $B_{1}$ is a $3$-cycle. Hence, $\left\{b_{1},b_{2},\hdots,b_{u}\right\}$ invariably generates $G$. If $K=\langle y\rangle$, with $|y|=3$, and $G$ centralises $y$, then $\left\{b_{1},b_{2},\hdots,yb_{u}\right\}$ invariably generates $G$. Finally, if $K=\langle y\rangle$, with $|y|=3$, and $G$ does not centralise $y$, then $G\cong 3^{u}:2$, with $Z(G)=1$, and the claim follows.

Finally, assume that $G/K$ is not of the form $3^{u-1}:2$ with trivial centre. Then, by induction, $G/K$ is invariably generated by a set $\left\{Kb_{1},Kb_{2},\hdots,Kb_{u-1}\right\}$. If $K$ is cyclic then the result is clear, so assume that $K\cong S_{3}$. Then we may assume that $b_{1}$ projects onto a $3$-cycle in $B_{1}$. Hence, taking $b_{u}$ to be any $2$-cycle in $K$, it is clear that $\left\{b_{1},b_{2},\hdots,b_{u-1},b_{u}\right\}$ invariably generates $G$. The result follows.\end{proof}

The following is easily checked using direct computation.
\begin{Proposition}\label{2less3Prop} Let $G\le GL_{n}(2)$ be irreducible, and assume that $n\le 4$ and $\di(G)>n/2$. Then either $n=2$ and $G\cong GL_{2}(2)\cong S_{3}$; or $n=3$ and $G\cong 7:3$ or $L_{3}(2)$; or $n=4$ and $G\cong Sp_{4}(2)\cong S_{6}$.\end{Proposition}

\begin{Lemma}\label{Subd1Lemma} Let $H$ and $K$ be groups with no common nontrivial homomorphic image, and let $G$ be a subdirect product of $H\times K$. Then $\di(G)\le \max\left\{\di(H),\di(K)\right\}$.\end{Lemma}
\begin{proof} The proof is almost identical to the proof of Lemma 2.4 in \cite{derek}; we give the details here for the readers benefit. Let $\left\{h_{1},h_{2},\hdots,h_{m}\right\}$ and $\left\{k_{1},k_{2},\hdots,k_{n}\right\}$ be invariable generating sets for $H$ and $K$ respectively, and assume, without loss of generality, that $m\ge n$. Also, let $g_{1}$, $g_{2}$,$\hdots$,$g_{m}\in G$, with $g_{i}=x_{i}y_{i}$ ($x_{i}\in H$, $y_{i}\in K$), and let $X=\langle h_{1}^{x_{1}}k_{1}^{y_{1}},\hdots,h_{n}^{x_{n}}k_{n}^{y_{n}},h_{n+1}^{x_{n+1}},\hdots,h_{m}^{x_{m}}\rangle$. We claim that $X=H\times K=G$. Indeed, if $R$ is any subdirect product of $H\times K$, then $R/[(R\cap H)\times(R\cap K)]$ is a homomorphic image of both $H$ and $K$. Thus, we must have $R=(R\cap H)\times (R\cap K)$, and hence $R=H\times K$. Since both $X$ and $G$ are subdirect products of $H\times K$, our claim follows.\end{proof}

\begin{Proposition}\label{Subd2Prop} Let $G$ be a subdirect product of $H\times K$, where $1\neq K$ is a finite group, and $H$ is isomorphic to either $S_{3}$, $S_{6}$, a nontrivial semidirect product $7:3$, or the group $L_{3}(2)$. Then\begin{enumerate}[(i)]
\item If $H\neq S_{6}$, then $\di(G)\le \di(K)+1$, and if $H=S_{6}$, then $\di(G)\le \di(K)+2$.
\item Suppose that $H=S_{3}$. If $K=7:3$ or $L_{3}(2)$, then $\di(G)=2$, while if $K=S_{6}$, then $\di(G)=3$.
\item Suppose that $H=S_{6}$. If $K=7:3$ or $L_{3}(2)$, then $\di(G)=3$, while if $K=S_{6}$, then $\di(G)\le 4$.
\item If $H\le GL_{3}(2)$ is irreducible, and $K$ is isomorphic to a semidirect product $3^{t}:2$, in which the involutions are self-centralising, then $\di(G)=t+1$.
\item If $H\le GL_{4}(2)$ is irreducible, and $K$ is isomorphic to a semidirect product $3^{t}:2$, in which the involutions are self-centralising, then $\di(G)=t+2$.
\end{enumerate}
\end{Proposition}
\begin{proof} We first prove (i). Since $G\le H\times K$ is subdirect, $G\cap H$ is a normal subgroup of $H$. If $G\cap H$ is cyclic, then $\di(G)\le \di(K)+1$, as needed. 

So assume that $G\cap H$ is noncyclic, and suppose first that $H\not\cong S_{6}$. Then $G=H\times K$; let $\left\{k_{1},\hdots,k_{t}\right\}$ be an invariable generating set for $K$. If $H=L_{3}(2)$, then take $x$ to be any element of $H$ of order $7$, and take $y$ to be any element of $H$ of order $4$; if $H=7:3$, then take $x$, $y\in H$ of order $3$ and $7$ respectively; and if $H=S_{3}$, then take $x$, $y\in H$ of order $2$ and $3$ respectively. Then $\left\{x,yk_{1},\hdots,k_{t}\right\}$ invariably generates $G$. Indeed, if $H=L_{3}(2)$, then $H=\langle x,x^{y}\rangle$, since $y\not\in N_{H}(\langle x\rangle)$, and the only proper subgroup of $H$ containing $x$ is of the form $7:3$; if $H=7:3$ or $S_{3}$, then clearly $H=\langle x,x^{y}\rangle$. Since our choice of elements, in each case, depended only on the orders, replacing $x$ and $y$ by any $H$-conjugates yields the same result. 

Assume now that $H=S_{6}$. If $G\cap H<S_{6}$, then $G\cap H=A_{6}$, and the result follows from Lemma \ref{diMin} part (i). So assume that $G=H\times K$, and let $x$, $y$, $z\in H$ be a $6$-cycle, a $5$-cycle and a $3$-cycle, respectively. Then $\left\{x,y,z\right\}$ is an invariable generating set for $H$. Also, let $g$, $h$, $g_{1}$, $g_{2}$, $\hdots$, $g_{t}\in G$. Note that $X:=\langle x^{g},y^{h},(zk_{1})^{g_{1}},k_{2}^{g_{2}}\hdots,k_{t}^{g_{t}}\rangle\le G$ is a subdirect product of $G$, with $X\cap H>1$ and $X\cap\neq A_{6}$ (since $X\cap H$ contains a $6$-cycle). Since $X\cap H\unlhd H$, we have $X\cap H=S_{6}$, so $X=G$ and part (i) follows.

Next, we prove (ii). So assume that $H=S_{3}$, and take $h_{1}$, $h_{2}\in H$ of orders $2$ and $3$ respectively. Assume first that $K=L_{3}(2)$ or $7:3$. It is easily seen that the only subdirect product of $H\times K$, in each case, is the full direct product. So $G=H\times K$. If $K=L_{3}(2)$, then choose $x$ and $y$ in $K$ of orders $7$ and $4$ respectively, and if $K=7:3$ then choose $x$ and $y$ in $K$ of orders $3$ and $7$ respectively. Then one can easily see that $\left\{h_{1}x,h_{2}y\right\}$ is an invariable generating set for $G$, which gives us what we need. So assume that $K=S_{6}$, and let $x$, $y$ and $z$ be cycles in $H$ of length $3$, $5$ and $6$ respectively. Then it is easy to see that $x$, $y$ and $z$ invariably generate $K$. If $G\cap H=1$, then the result is clear. If $G\cap H=A_{3}$, then take $X=\left\{x,h_{2}y,z\right\}$, and if $G\cap H=H$, then take $X=\left\{x,h_{2}y,z\right\}$. Clearly $X$ invariably generates $G$ in each case, and this completes the proof of (ii).

We now consider (iii). So $H=S_{6}$. If $K=7:3$ or $L_{3}(2)$ then the result follows from Proposition \ref{Subd1Lemma}, so assume that $K=S_{6}$. If $G\cap H=1$, then $\di(G)=\di(K)=3$, so assume that $G\cap H=A_{6}$ or $S_{6}$. Let $k_{1}$, $k_{2}$, $k_{3}\in G$ such that $(G\cap H)k_{1}$ is a $3$-cycle, $(G\cap H)k_{2}$ is a $5$-cycle, and $(G\cap H)k_{3}$ is a $6$-cycle in $G/G\cap H\cong S_{6}$. Then $\left\{(G\cap H)k_{1},(G\cap H)k_{2},(G\cap H)k_{3}\right\}$ is an invariable generating set for $G/G\cap H$. Now, since $A_{6}\le G\cap H$, we my assume, by replacing $k_{1}$, $k_{2}$ by suitable powers, that $k_{1}$ and $k_{2}\in G$ project onto a $5$-cycle and a $3$-cycle in $H$, respectively. If $G\cap H=H$, then choose $h\in H$ to be a $6$-cycle in $H$; otherwise, set $h:=1\in G\cap H$. Then $\left\{k_{1},k_{2},k_{3},h\right\}$ is an invariable generating set for $G$, and this proves (iii).
 
Finally, we prove (iv) and (v). So assume that $H\le GL_{n}(2)$ is irreducible, where $n=3$ or $4$, and that $K\cong 3^{t}:2$ with $Z(K)=1$. If $H$ and $K$ have no common nontrivial homomorphic images, then the result follows from Proposition \ref{Subd1Lemma}, so assume otherwise. Then, since the only normal subgroups of $K$ are the $3$-subgroups, we conclude that $H$ must have a nontrivial homomorphic image of order $3^{j}2$, where $j\ge 0$. Using the database of irreducible matrix groups in MAGMA, we have $n=4$, and $G\cap H\unlhd H$ contains elements $x$, $y$ and $z$ with $(|x|,2)=(|y|,3)=1$, such that $\left\{x,y,z\right\}$ invariably generates $G\cap H$. Choose $k_{1}$, $\hdots$, $k_{t}$, $k_{t+1}\in G$, such that $|(G\cap H)k_{i}|=3$ for $1\le i\le t$, $|(G\cap H)t_{i}|=2$, and $\left\{(G\cap H)k_{1},(G\cap H)k_{2},\hdots,(G\cap H)k_{t},(G\cap H)k_{t+1}\right\}$ invariably generates $G/G\cap H\cong K$. Also, let $L:=(G\cap H)\times (G\cap K)$, and suppose first that $L=G$. Then $G=H\times K$, since $G$ is subdirect, and hence we may assume that the $k_{i}$ are elements of $K$. It now follows easily that $\left\{z,yk_{1},k_{2},\hdots,k_{t}xk_{t+1}\right\}$ invariably generates $G$, which gives us hat we need.

So we may assume that $L<G$. Hence, $G/L$ is a common nontrivial homomorphic image of $H$ and $K$. As mentioned above, we must have $|G:L|=3^{j}2$, some $j\ge 0$. If $G\cap H$ is cyclic, then $\di(G)\le \di(K)+1=t+2$, so assume also that $G\cap H$ is a noncyclic normal subgroup of $H$. By direct computation, $G/L\cong H/G\cap H\cong K/G\cap K$ is isomorphic to either $C_{2}$ or $S_{3}$. Furthermore, apart from the case $(H,G\cap H)=(3^{2}:2,3^{2})$, where $|Z(H)|=3$, for each other pair $(H,G\cap H)$ satisfying these conditions, by direct computation we may choose an invariable generating set $\left\{x,y\right\}$ for $G\cap H$ such that $(|x|,3)=1$. Then if $\left\{g_{1},\hdots,g_{t-j}\right\}$ is an invariable generating set for $G\cap K$, then $\left\{y,xg_{1},\hdots,g_{t-j}\right\}$ is an invariable generating set for $L$. Also, since $j=0$ or $1$, it is clear that $\di(G/L)\le j+1$. Hence $\di(G)\le \di(G/L)+\di(L)\le t+2$. So we may assume that $(H,G\cap H)=(3^{2}:2,3^{2})$, where $|Z(H)|=3$; choose an invariable generating set $\left\{x,y\right\}$ for $G\cap H$ with $Z(G\cap H)=\langle y\rangle$. Also, take $k_{1}$, $k_{2}$, $\hdots$, $k_{t+1}$ to be as in the previous paragraph; by replacing $k_{t+1}$ by $k_{t+1}^{3}$ if necessary, we may assume that $k_{t+1}$ projects onto an element of order $2$ in $H$. Hence, $\left\{x,k_{1},k_{2},\hdots,yk_{t+1}\right\}$ is an invariable generating set for $G$, and this completes the proof.\end{proof}

Next, we consider the function $\di$ on direct products of nonabelian simple groups. Wiegold proves in \cite[Lemma 2]{Wie} that if $r$ is the number of direct factors in such a group $G$, then $d(G)\le 2+\lceil \log_{60}{r}\rceil$. However, in \cite{KLS}, it is shown that this bound fails when one replaces $d$ by $\di$. In fact, for each positive integer $n$, there is a $2$-generated group $G$ (which is a direct product of isomorphic nonabelian simple groups), such that $\di(G)>n$. What we do have, however, is the following.
\begin{Theorem}[{\bf\cite{KLS} Theorem 5.1}]\label{KLSSimple} Let $T$ be a nonabelian simple group.\begin{enumerate}[(a)]
\item If $T$ is not one of the groups $P\Omega^{+}(8,2)$ or $P\Omega^{+}(8,3)$, then there are two elements $s$, $t\in T$ such that $T=\langle s^{g_{1}},t^{g_{2}}\rangle$ for each choice of $g_{i}\in \Aut{(T)}$.
\item If $T=P\Omega^{+}(8,2)$ or $P\Omega^{+}(8,3)$, and if $T\le A\le \Aut{(T)}$, then there are elements $t\in T$, $s\in A$, such that $T\le \langle s^{g_{1}},t^{g_{2}}\rangle$ for each choice of $g_{i}\in A$.\end{enumerate}\end{Theorem}

We now prove a consequence of Theorem \ref{KLSSimple}.
\begin{Corollary}\label{KLSSimpleCor} Let $G=T_{1}\times T_{2}\times\hdots\times T_{r}$ be a direct product of isomorphic nonabelian simple groups $T_{i}$. If $r$ is even then $\di(G)\le r$, and if $r$ is odd then $\di(G)\le r+1$.\end{Corollary}
\begin{proof} By Theorem \ref{KLSSimple}, our claim will follow if we can prove that $\di(T_{1}\times T_{2})=2$. So write $T=T_{1}$ and assume that $G=T^{2}$. If $T$ is not $P\Omega^{+}(8,2)$ or $P\Omega^{+}(8,3)$, then let $s,t\in T_{1}$ be as in Theorem \ref{KLSSimple}; if $T=P\Omega^{+}(8,2)$, then by direct computation we can choose $s,t\in T$ such that $|s|\neq |t|$, and $T$ is invariably generated by $\left\{s,t\right\}$. We claim that $\left\{(s,t),(t,s)\right\}$ invariably generates $G$. To see this, let $(g_{1},g_{2})$, $(h_{1},h_{2})\in G$, and set $X:=\langle (s^{g_{1}},t^{g_{2}}),(t^{h_{1}},s^{h_{2}})\rangle$. Let $\rho_{i}:G\rightarrow T_{i}$ be the projection maps. Then $\rho_{i}(X)=T_{i}$ for each $i$ by Theorem \ref{KLSSimple}, so either $X=G$, or $X$ is a diagonal subgroup of $G$.

So assume that $X$ is a diagonal subgroup of $G$. Then $X=\left\{(y,y^{\alpha})\text{ : }y\in T_{1}\right\}$, for some $\alpha\in \Aut{(T_{1})}$. But then, in particular, $t=s^{g_{1}\alpha g_{2}^{-1}}$. This contradicts part (a) of Theorem \ref{KLSSimple} in the case $T\neq P\Omega^{+}(8,2)$, $P\Omega^{+}(8,3)$, while the contradiction is clear if $T$ is one of $P\Omega^{+}(8,2)$, $P\Omega^{+}(8,3)$; indeed, $|s|\neq |t|$ in these cases. The result now follows.\end{proof}

We now turn to the symmetric group.
\begin{Proposition}\label{SymmetricGroup}
\begin{enumerate}[(i)]
\item Let $G\cong S_{n}$. Then $\di(G)\le 2$, except that $\di(S_{6})=3$;
\item  Let $A\cong \Aut{(A_{6})}$, $P\Gamma L_{2}(8)$, $PGL_{2}(7)$, $PGL_{4}(3)$ or $L_{3}(4).m$, where $m=2$, $3$ or $6$. Then $\di(A)\le 2$.\end{enumerate}\end{Proposition} 
\begin{proof} \begin{enumerate}[(i)]\item If $n\le 7$ then the result is easy, so assume that $n\ge 8$. Then we can choose a prime $p$ with $n/2<p<n-2$ (this is clear for $8\le n\le 11$, and follows from \cite[Theorem 1.3]{elba} if $n\ge 12$); in particular, $p$ is odd. Let $x$ be any $n$-cycle in $G$, and let $y$ be any $p$-cycle. If $n$ is even, then set $z:=1\in G$; otherwise, let $z$ be any transposition in $G$. Then $G=\langle x,yz\rangle$ by Lemma 8.20 and Theorem 8.23 in \cite{IsaacsFGT}. This completes the proof.
\item We prove the claim by direct computation, using MAGMA. In each of the cases $A=\Aut{(A_{6})}$, $P\Gamma L_{2}(8)$, $PGL_{2}(7)$, $PGL_{4}(3)$, $L_{3}(4).2$, $L_{3}(4).3=PGL_{3}(4)$ or $L_{3}(4).6$, $A$ has two elements $x$ and $y$, which invariably generate $G$. The lists $[|x|,i(x);|y|,i(y)]$, where $i(t)$ denotes the conjugacy class number of the element $t$ of $A$ in MAGMA, is $[8,11;10,13]$, $[7,8; 9,11]$, $[6,6;8,9]$, $[26,39;40,43]$, $[5,8; 14,14]$, $[15,18;21,22]$ and $[8,14;21,20]$ respectively.\end{enumerate}\end{proof} 

\begin{Proposition}\label{PrimIrr43Prop} Assume that Theorem \ref{CompRedTheorem} Part 1 holds for $F=\mathbb{F}_{2}$, and let $G\le GL_{4}(3)$ be primitive and irreducible. Then $\di(G)\le 4$.\end{Proposition}
\begin{proof} Let $f=f(G)$, $K_{1}:=K_{1}(G)$. Then $G\le GL_{4/f}(2^{f}).f$, $K_{1}:=G\cap GL_{4/f}(2^{f})$ is irreducible and weakly quasiprimitive, and all characteristic abelian subgroups of $K_{1}$ lie in the scalar subgroup $Z$ of $K_{1}$. Using Lemma \ref{FStarLemma} and Corollary \ref{FStarCor}, we know the possible structures of $K_{1}/Z$; we now consider each case. 

Suppose first that $G$ has shape $E.X$, where $E$ is extraspecial of order $2^{1+4}$, and $X\le Sp_{4}(2)$ is completely reducible. Let $u$ be an element of $E$ of order $4$, so that $Z=\langle u^{2}\rangle$, and let $U=\langle u\rangle^{G}$. Assume first that $G$ does not centralise $U/Z$. Then $U/Z$ is a $G$-submodule of $E/Z$ of dimension $d$ at least $2$. Let $Z\le V\le E$ such that $V/Z$ is a $G$-complement for $U/Z$ in $E/Z$. If $d=3$ or $4$, or if $G$ does not centralise $V/Z$, then $\langle U,x\rangle^{G}=E$ for some $x\in V$ and hence, since $\di(X)\le 2$ by Theorem \ref{CompRedTheorem}, we have $\di(G)\le 4$. So assume that $G/Z$ acts trivially on $V/Z$. Then $G/U\cong 2^{2}.S_{3}$, where the action of $S_{3}$ on $2^{2}$ is trivial. It is easy to see that such a group $G/U$ can be invariably generated by $3$ elements, and hence $\di(G)\le 4$, as needed.

So we may assume that $G$ centralises $U/Z$. Thus, $\dim{(V/Z)}=3$. If $V/Z$ is irreducible as a $G$-module, then $\di(G)\le 4$ again, so assume that $V/Z$ is reducible. Then clearly $V/Z$ is a direct sum $V/Z=S/Z\times T/Z$, where $S/Z$ is a trivial $G$-module, and $T/Z$ is a $2$-dimensional $G$-module. Hence, $G/V\cong (E/V).(G/E)$, and $G/E\cong S_{3}$, acting trivially on $E/V\cong U/Z\cong 2$. Now, let $x$, $y\in G$ such that $|x|=3$, $Ex$ is a $3$-cycle in $G/E$, and $Ey$ is a $2$-cycle. Finally, let $v_{1}$, $v_{2}\in V$ such that $V/Z=\langle Zv_{1},Zv_{2}\rangle^{G}$. Then $\left\{ux,v_{1},v_{2},y\right\}$ in an invariably generating set for $G$, and this completes the proof in the case $G\le 2^{1+4}.Sp_{4}(2)$.      

When $f=1$, the only remaining possibility is $f=1$ and $G/Z$ is almost simple. Then, using the database of irreducible matrix groups in MAGMA, we have $(G/Z)\le S_{5}$, $A_{6}.2^{2}$, or $L_{4}(3).2=PGL_{4}(3)$. Each of these groups is invariably $2$-generated by Proposition \ref{SymmetricGroup}, so $\di(G)\le 4$.

So we may assume that $f=2$. Then $K_{1}/Z\le S_{4}$ or $K_{1}/Z$ is almost simple. Since $G/K_{1}$ is cyclic, it suffices to prove that $\di(K_{1}/Z)\le 2$, and this is clear when $K_{1}/Z\le S_{4}$. So assume that $K_{1}/Z$ is almost simple. Again using the database of irreducible matrix groups in MAGMA, we have $(K_{1}/Z)\le S_{5}$, $A_{6}.2^{2}$, or $L_{4}(3).2=PGL_{4}(3)$. The claim now follows from Proposition \ref{SymmetricGroup}, and this completes the proof.\end{proof} 

\section{Completely reducible linear groups and the proof of Theorem \ref{CompRedTheorem}}\label{ProofSection1}
Before proceeding to the proof of Theorem \ref{CompRedTheorem}, we need a lemma which is analogous to \cite[Lemma 4.1]{derek}. We remark that our proof follows the same strategy of the proof of the afore mentioned lemma in \cite{derek}. 
\begin{Lemma}\label{Lemma4.1} Let $G$ be a finite group, with a normal elementary abelian subgroup $N$ of order $p^{m}$, such that $N=C_{G}(N)$, and the action of $G/N$ on $N$ is completely reducible. Let $H$ be a subnormal subgroup of $G$, and assume that Theorem \ref{CompRedTheorem} holds for $F=\mathbb{F}_{p}$ and dimensions $n\le m$. Then \begin{enumerate}[(i)]
\item If $p=2$, then $\di(H)\le m$;
\item If $p=3$, then $\di(H)\le 3m/2$ if $m>1$, and $\di(H)\le 2$ if $m=1$;
\item If $p>3$, then $\di(H)\le 2m$.\end{enumerate}\end{Lemma}
\begin{proof} As shown in the first paragraph of the proof of Lemma 4.1 in \cite{derek}, it suffices to prove the result for groups $G$ as in the statement of the lemma, rather than their subnormal subgroups. 

Let $M$ be the direct sum of the one-dimensional $G$-submodules of $N$, so that $|M|=p^{l}$, for some $l\le m$.

Suppose first that $p=2$. Then $M$ is a trivial $G$-module, so $M=Z(G)$ (since $N$ is self-centralising). It follows that $G/N$ acts faithfully on $N/M$, and hence, by complete reducibility, and since $\di(M)=l$, it suffices to prove the result for $G/M$. That is, we may assume that $M=1$. So all irreducible constituents of $N$ have dimension at least $2$, and hence $d_{G}(N)\le m/2$. The result now follows from Theorem \ref{CompRedTheorem}, unless $G/N\cong B_{m}$, where $B_{m}=3^{m/2}:2\le GL_{2}(2)^{m/2}$ is as defined in Theorem \ref{CompRedTheorem}. In this case, let $g_{1}$, $g_{2}$, $\hdots$, $g_{m/2}$, $g$ be elements of $G$, with the $g_{i}$ of $3$-power order, and $|Ng|=2$, such that $G/N$ is invariably generated by $Ng_{1}$, $\hdots$, $Ng_{m/2}$, $Ng$. Now, choose a generating set $x_{1}$, $\hdots$, $x_{m/2}$ for $N$ as a $G$-module, with $x_{1}\in C_{G}(g_{1})$. We claim that the set $\left\{g_{1}x_{1},g_{2},\hdots,g_{m/2},g,x_{2},\hdots,x_{m/2}\right\}$ invariably generate $G$. So let $h$, $h_{1}$, $h_{2}$, $\hdots$, $h_{m/2}$, $y_{2}$ ,$\hdots$, $y_{m/2}$ be elements of $G$. We need to prove that $G=H:=\langle (g_{1}x_{1})^{h_{1}},g_{2}^{h_{2}},\hdots,g_{m/2}^{h_{m/2}},g^{h},x_{2}^{y_{2}},\hdots,x_{m/2}^{y_{m/2}}\rangle$. Now, $x_{1}^{h_{1}}=[(g_{1}x_{1})^{h_{1}}]^{|g_{1}|}\in H$, and $(g_{1}^{2})^{h_{1}}=[(g_{1}x_{1})^{h_{1}}]^{2}\in H$ since $x_{1}$ centralises $g_{1}$, and $g_{1}$ has $3$-power order. Since $\left\{Ng_{1}^{2},Ng_{2},\hdots,Ng_{m/2},Ng\right\}$ is an invariable generating set for $G/N$, and $\left\{x_{1},x_{2},\hdots,x_{m/2}\right\}$ generates $N$ as a $G$-module, the claim follows, and hence $\di(G)\le m$, as needed.

So assume now that $p>2$, and let $L$ be a $G$-submodule of $N$ which complements $M$. Let $C:=C_{G}(L)$. Then $\di(G/C)\le \epsilon(m-l)$, by Theorem \ref{CompRedTheorem}, where $\epsilon=\epsilon(p):=1$ if $p=3$ and $\epsilon:=3/2$ if $p>3$. Now, $C/N$ acts faithfully on $N/L\cong M$, so $C/N\le GL(M)\cong (p-1)^{l}$. Hence, $\di(C/L)\le \di(C/N)+\di(M)\le 2l$. Also, if $p=3$ then $C/L$ is isomorphic to a subgroup of $3^{l}:2^{l}\cong \Sym{(3)}^{l}$, so Proposition \ref{S3Prop} implies that $\di(C/L)\le l+1$ in that case. Thus, since $d_{G}(L)\le (m-l)/2$, we have $\di(G)\le (\epsilon(p)+1/2)(m-l)+l+\delta(p)$, where $\delta(p):=1$ if $p=3$ and $\delta(p):=l$ otherwise. The result now follows, except when $p=3$ and $l\le 1$. If $l=0$, then $L=N$ and $d_{I}(G)\le d_{G}(L)+\di(G/N)\le 3m/2$ follows from Theorem \ref{CompRedTheorem}. If $l=1$ and $C/L$ is cyclic, then we're done. So we may assume that $C/L\cong S_{3}$. Now let $Y$ be a subset of $G$, of size $m-1$, whose image modulo $C$ invariably generates $G/C$; let $x_{1}$, $x_{2}$, $\hdots$, $x_{(m-1)/2}$ be a generating set for $L$ as a $G$-module; let $\sigma$ be an element of $C$ which reduces to a $3$-cycle modulo $L$; and let $\tau$ be an element of $C$ of order $2$. It is now easy to see that $Y\cup\left\{\sigma,x_{1}\tau,x_{2},\hdots,x_{(m-1)/2}\right\}$ is an invariable generating set for $G$, and this gives us what we need.\end{proof} 

\begin{Proposition}\label{dOuterSimple} Let $T$ be a nonabelian finite simple group, and let $H$ be a subgroup of $\Out{(T)}$. Then $d_{I}(H)\le 3$.\end{Proposition}
\begin{proof} The structures of the outer automorphism groups of the finite simple groups are well known. When $T$ is an alternating group, $|\Out{(T)}|\le 4$. When $T$ is a simple classical group (see \cite[Chapter 2]{KleidLie}), $\Out{(T)}$ modulo its (cyclic) group $C$ of field automorphisms is either metacyclic ($T\neq P\Omega^{+}_{8}(q)$, for $q$ odd), or isomorphic to $S_{4}$ ($T=P\Omega^{+}_{8}(q)$). Finally, if $T$ is an exceptional group then \cite[Table 5.1.B]{KleidLie} implies that $|\Out{(T)}/C|\le 6$, while if $T$ is a sporadic group then $|\Out{(T)}|\le 2$, using \cite[Table 5.1.C]{KleidLie}. The result now follows in each case.\end{proof}

The preparations are finally complete. We will prove both parts of Theorem \ref{CompRedTheorem} together by induction on $n$. If $n=1$ then $G$ is cyclic, and Parts 1 and 2 clearly hold. So assume, below, that $n>1$, and that Theorem \ref{CompRedTheorem} holds for dimensions less than $n$. We may also assume that for fixed $n$, the theorem holds for fields of order less than $|F|$.
\begin{proof}[Proof of Part 2 of Theorem \ref{CompRedTheorem}] Since $R$ is homogeneous, it acts faithfully on each of its irreducible constituents, so we may assume that $R$ is irreducible. Let $f:=f(R)$ and $K_{1}:=K_{1}(R)$ be as defined after Lemma \ref{DefLemma}, and let $K$ and $F_{1}$ be as in Lemma \ref{DefLemma}, so that $K_{1}\cong K\le R$. If $f>1$, then $K_{1}$ satisfies the inductive hypothesis, so $H\cap K_{1}$ modulo its scalar subgroup can be invariably generated by $2\log{(r/f)}$ elements. Also, $HK/K$ is abelian of order at most $f$. Hence, $\di(H)\le \log{f}+2\log{(n/f)}+1\le 2\log{n}$.

So we may assume that all characteristic abelian subgroups of $R$ are contained in $Z$. Let $q_{i}$, $m_{i}$, $t_{j}$, $T_{j}$, $u_{j}$ and $U_{j}$ be as in Corollary \ref{FStarCor}. Then $Z=Z(R)$ is the intersection of the groups $C_{R}(O_{q_{i}})$ and $C_{R}(U_{j})$ over all $i$, $j$. Thus, since $\log{x}+\log{y}=\log{xy}$, and since $\prod_{i} q_{i}^{m_{i}}\prod_{j} t_{j}^{u_{j}}$ divides $n$, it will suffice to prove that, for each $i$, $j$, each subnormal subgroup of $R/C_{R}(O_{q_{i}}(R))$ can be invariably generated by $\log{q_{i}^{m_{i}}}$ elements, and that each subnormal subgroup of $R/C_{R}(U_{j})$ can be invariably generated by $\log{s_{j}^{t_{j}}}$ elements.

To this end, we first consider a subnormal subgroup $H$ of $R/C_{R}(O_{q}(R))$, where $q$ is a prime such that $O_{q}(R)$ is not contained in $Z$. By Lemma \ref{FStarO}, $R/C_{R}(O_{q}(R))$ has shape $q^{2m}.X$, where $X$ is a completely reducible subgroup of $Sp_{2m}(q)$, for some $m\ge 1$. We need to prove that $\di(H)\le 2m\log{q}$. If $2m<n$, then since Theorem \ref{CompRedTheorem} (including Part 1) holds for dimensions less than $n$, we can apply Lemma \ref{Lemma4.1} and conclude that $\di(H)\le 2\delta m$, where $\delta:=1$ if $q=2$, $\delta:=3/2$ if $q=3$ and $\delta:=2$ otherwise. In particular, $\delta\le \log{q}$, which gives us what we need. So assume that $2m\ge n$. Since $2m\le q^{m}\le n$, we must have $q=2$ and $m\le 2$. If $m=1$, then $R/C_{R}(O_{q}(R))\le 2^{2}.S_{3}$, and hence every subnormal subgroup of $R/C_{R}(O_{q}(R))$ can be invariably generated by $2$ elements. If $m=2$, then $R/C_{R}(O_{q}(R))\le 2^{4}.Sp_{4}(2)$, and hence every subnormal subgroup of $R/C_{R}(O_{q}(R))$ can be invariably generated by $4$ elements, by Lemma \ref{Lemma4.1}

Next, let $H$ be a subnormal subgroup of $\overline{R}:=R/C_{R}(U)$, where $U$ is a central product of $u$ copies of a quasisimple group $T$, and $T$ has a faithful irreducible representation of degree $t$ over $F$. We need to show here that $\di(H)\le 2u\log{t}$. By Lemma \ref{FStarT}, $\overline{T}^{u}\le\overline{R}\le \overline{T}^{u}.(A\wr\Sym{(u)})$, where $A\le \Out{(\overline{T})}$. Furthermore, $|A|\le 2$ if $t=2$. Since $H\cap \overline{T}^{u}$ is a subnormal subgroup of $\overline{T}^{u}$, it follows that $H\cap\overline{T}^{u}\cong \overline{T}^{v}$ for some $v\le u$. If $u=1$, then $\overline{R}\le \Aut{(\overline{T})}$, so $\di(H)\le 5$ by Proposition \ref{dOuterSimple} and Corollary \ref{KLSSimpleCor}. Using Corollary \ref{KLSSimple} and Proposition \ref{dOuterSimple}, we get $\di(H)\le 2+6+1=9$ if $u=2$, and $\di(H)\le 4+9+2=15$ if $u=3$. If $u\ge 4$, then Proposition \ref{dOuterSimple}, Corollary \ref{KLSSimpleCor} and Theorem \ref{DetLucHalfn} imply that $\di(H)\le u+1+3u+u/2=9u/2+1<5u$. If $t>5$, then $5<2\log{t}$, and the result follows, so assume that $t\le 5$. Suppose first that $t=2$. Then $H/\Soc{(H)}\le 2\wr\Sym{(u)}$, and hence either $u$ is even, or $d_{I}(H/\Soc{(H)})\le u-1$, by Proposition \ref{2Prop}. If $u$ is even, then $\di(\Soc{(H)})\le u$ by Corollary \ref{KLSSimpleCor}. Otherwise, $\di{(\Soc{(H)})}\le u+1$. Thus, in either case, we have $\di(H)\le 2u$, as needed.  

Suppose next that $t=3$. Then $|A|\le 3$ by \cite{12,19}, so $\di(H)\le u+1+3u/2+\delta_{3,u}/2$ by Theorem \ref{DetLucHalfn}, and this is less than $(2\log{3})u$, except when $u=1$. But in this case $H\le \overline{T}.A$, so $\di(H)\le 3<2\log{3}$ by Corollary \ref{KLSSimpleCor}, since $A$ is cyclic, which gives us what we need. Assume now that $t=4$. Then $|A|=1,2$ or $4$ by \cite{12,19}, so $H/\Soc{(H)}\le \Sym{(4)}\wr\Sym{(u)}\le \Sym{(4u)}$, and hence $\di(H)\le 3u+1\le 4u$ using Theorem \ref{DetLucHalfn} and Corollary \ref{KLSSimpleCor}. Finally, suppose that $t=5$. Then by \cite{12,19}, $A$ is metacyclic, so every subgroup of $A$ can be invariably generated by $2$ elements. Hence, if $u=1$ then $\di(H)\le 4<2\log{5}$, and if $u>1$ then $\di(H)\le u+1+2u+u/2+1/2\le 9u/2<2u\log{5}$, again using Corollary \ref{KLSSimple} and Theorem \ref{DetLucHalfn}, as needed.\end{proof}

\begin{proof}[Proof of Part 1 of Theorem \ref{CompRedTheorem}] Define $\epsilon=\epsilon(F)$ to be $1/2$, $1$ or $3/2$ according to whether $|F|=2$, $|F|=3$ or $|F|>3$, respectively. We need to prove that $\di(G)\le \epsilon n$.

Suppose first that $G$ is reducible, let $U$ be an irreducible submodule of the natural module $V$, and let $W$ be a $G$-complement for $U$ in $V$. Then $G$ embeds as a subdirect product of $G^{U}\times G^{W}$, where $G^{U}$, $G^{W}$ denote the induced actions of $G$ on $U$ and $W$ respectively. Since the embedding is subdirect, $G/G\cap G^{W}$ is isomorphic to $G^{U}$, and $G\cap G^{W}$ is a normal subgroup of $G^{W}$ (here, we are viewing $G$ as a subgroup of $G^{U}\times G^{W}$. Thus, $G\cap G^{W}$ is completely reducible by Clifford's Theorem. Since $d_{I}(G)\le d_{I}(G/G\cap G^{W})+d_{I}(G\cap G^{W})$, the result now follows from the inductive hypothesis, except when $|F|=2$ and $G$ has irreducible constituents of dimension $2$, $3$ or $4$. So assume that $|F|=2$. 

Suppose first that $\dim{U}=2$, so that $G^{U}\le S_{3}$, and $\dim{W}=n-2$. If $\di(G^{W})\le \lfloor \frac{n-2}{2}\rfloor$, then $\di(G)\le \frac{n-2}{2}+1=\frac{n}{2}$, as needed, by Proposition \ref{Subd2Prop} (i). So assume that $\di(G^{W})>\frac{n-2}{2}$. Then Proposition \ref{2less3Prop}, together with the inductive hypothesis, implies that $G^{W}$ is isomorphic to either $S_{3}$, $B_{n-2}$, $7:3$, $L_{3}(2)$ or $Sp_{4}(2)\cong S_{6}$. If $G^{W}\cong S_{3}$, then the result follows from Proposition \ref{S3Prop}. If $G^{W}$ is isomorphic to $7:3$ or $L_{3}(2)$, then $\di(G)\le 2$, while if $G^{W}\cong Sp_{4}(2)$, then $\di(G)\le 3$, by Proposition \ref{Subd2Prop} (ii). So assume that $G^{W}\cong B_{n-2}$. If $G$ is the full direct product $G^{U}\times G^{W}$, then the result follows from Proposition \ref{S3Prop}. So assume that $G\cap G^{U}<S_{3}$. If $|G\cap G^{U}|=1$ then $\di(G)=\di(K)=\frac{n}{2}$. Otherwise, $|G\cap G^{U}|=3$, so $G\cong B_{n}$ and $\di(G)=\left\lfloor \frac{n}{2}\right\rfloor+1$.

Assume now that $\dim{U}=3$. Then $G^{U}\cong 7$, $7:3$ or $GL_{3}(2)$. If $\di(G^{W})\le \left\lfloor \frac{n-3}{2}\right\rfloor$, then the result follows from Proposition \ref{Subd2Prop} (i). So assume that $G^{W}$ is isomorphic to one of $S_{3}$, $7:3$, $L_{3}(2)$, $Sp_{4}(2)$, or $B_{n-3}$. The result then follows, from Lemma \ref{Subd1Lemma} if $G^{W}\cong S_{3}$, $Sp_{4}(2)$; Proposition \ref{Subd2Prop} (i) if $G^{W}\cong 7:3,GL_{3}(2)$; and from part Proposition \ref{Subd2Prop} (iv) otherwise.

Finally, suppose that $\dim{U}=4$. If $\di(G^{W})\le \lfloor \frac{n-4}{2}\rfloor$, or if $G^{W}\cong B_{n-4}$, then the result follows, from Proposition \ref{Subd2Prop} parts (i) and (v). So assume that $G^{W}\cong S_{3}$, $7:3$, $L_{3}(2)$ or $Sp_{4}(2)$. If $G^{U}\not \cong Sp_{4}(2)$, then $\di(G)\le \di(G^{U})+1=3$ by Propositions \ref{2less3Prop} and \ref{Subd2Prop} (i). If $G^{U}\cong Sp_{4}(2)$, then the required upper bound follows from Proposition \ref{Subd1Lemma} if $G^{W}\not\cong S_{3}$, $Sp_{4}(2)$, and Proposition \ref{Subd2Prop} parts (ii) and (iii) otherwise.

So we may assume that $G$ is irreducible. If $G$ is imprimitive, let $\Delta$ be a minimal block for $G$, of dimension $r$ say, let $R:=G_{\Delta}^{\Delta}$ be the induced action of the stabiliser $G_{\Delta}$ on $\Delta$, and let $S$ be the induced action of $G$ on the set of $G$-conjugates of $\Delta$. If $G$ primitive, set $\Delta:=V$, $r:=n$, $R:=G$ and $S:=1$. Since $G$ is irreducible, $R\le GL_{r}(F)$ is irreducible, and $S\le \Sym{(s)}$ is transitive, where $s:=n/r$. The minimality of $\Delta$ implies also that $R$ is primitive. In particular, each subnormal subgroup of $R$ can be invariably generated by $h(r)$ elements, where $h(r):=2\log{r}+1$ if $|F|>2$ or if $(r,|F|):=(4,2)$, and $h(r):=2\log{r}$ otherwise, by part 1 of Theorem \ref{CompRedTheorem}. Hence, $\di(G)\le h(r)s+(s+\delta_{s,3}/)2$, and this yields $\di(G)\le \epsilon rs$, except when $|F|=2$ and $r\le 17$; when $|F|=3$ and $r\le 7$; and when $|F|>3$ and $r\le 3$.

We deal with these exceptional cases as follows. Throughout, we write $f=f(R)$ and $K_{1}=K_{1}(R)$, and we take $F_{1}=F_{1}$ to be a degree $f$ extension of $F$, as in Lemma \ref{DefLemma}. We also identify $K_{1}$ as a characteristic subgroup of $R$ with $R/K_{1}$ abelian of order at most $f$ (see Lemma \ref{DefLemma}). If $f=r$, then $\di(G)\le 2s+\left\lfloor s/2\right\rfloor$, which is less than $rs/2$ when $r\ge 3$. Thus, in our case by case analysis below, we may assume that if $r\ge 3$, then $f<r$. \begin{enumerate}
\item $r=1$. If $G$ is primitive then $G$ is cyclic (or trivial when $|F|=2$), and the result is clear, so assume that $G$ is imprimitive. If $|F|=2$ then $R$ is trivial, and hence $G$ is reducible, a contradiction. So $|F|>2$. In any case, $R$ is cyclic, so Proposition \ref{Cmwr3} if $s=3$, or Theorem \ref{DetLucHalfn} otherwise, implies that $d_{I}(G)\le s+s/2=3s/2$. If $|F|=3$, then Corollary \ref{chief} yields $\di(G)\le E(s,2)+\di(S)$. Now, $E(s,2)\le s/2$ for all $s$, so by Theorem \ref{DetLucHalfn} we may assume that $s=3$. But in this case $\di(G)\le E(3,2)+2=3$, as needed.
\item $r=2$. Suppose first that $|F|=2$, so that $R\le S_{3}$. If $G$ is primitive, then $\di(G)\le 2$, which gives us what we need. So we may assume that $G$ is imprimitive. Then $\di(G)\le E(s,2)+E(s,3)+\di(S)$, by Corollary \ref{chief}. If $s=3$, and either $R<S_{3}$ or $S<S_{3}$, then we get $\di(G)\le E(3,2)+2$ or $\di(G)\le E(3,2)+E(3,3)+1$, which in each case yields $\di(G)\le 3$, as needed. If $s=3$ and $R=S=S_{3}$, then $G$ is a transitive subgroup of $\Sym{(9)}$, and the result follows from Proposition \ref{TransProp}. So assume that $s\neq 3$. Then $\di(S)\le s/2$, and the result follows from the bound $\di(G)\le E(s,2)+E(s,3)+\di(S)$ if $E(s,2)+E(s,3)\le s/2$. It is easily seen, from the definition of $E$, that $E(s,2)+E(s,3)>s/2$ only when $3\neq s$ is $2$, $4$, $6$ or $12$. However, when $s=2$, $4$ or $6$, then $G$ is transitive of degree $6$, $12$ or $18$, respectively (and $G<S_{6}$ when $s=2$), so the result follows from Proposition \ref{TransProp}. When $s=12$, $\di(S)\le 4$ by the same proposition, and hence $\di(G)\le E(12,2)+E(12,3)+4=11$, which gives us what we need.

Next, assume that $|F|=3$. Then $R\le GL_{2}(3)$. If $G$ is primitive, then $G$ is either cyclic or quaternion of order $8$, or isomorphic to $SD_{16}$, $SL_{2}(3)$ or $GL_{2}(3)$. All of these groups are easily seen to be invariably $2$-generated, so we may assume that $G$ is imprimitive. Then $\di(G)\le 4E(s,2)+E(s,3)+\di(S)$, by Corollary \ref{chief}. Since $E(3,2)=E(3,3)=1$, the case $s=3$ follows if either $R<GL_{2}(3)$ or $S<S_{3}$. Thus we may assume that if $s=3$, then $R=GL_{2}(3)$ and $S=S_{3}$. Then $G/G\cap Z^{s}\le S_{4}\wr S_{3}$ is transitive of degree $12$, and hence can be invariably generated by $4$ elements, by Proposition \ref{TransProp}. Thus, $\di(G)\le E(3,2)+4=5$, which gives us what we need. So assume that $s\neq 3$. If $s$ is not $2$, $4$, $6$, $8$ or $16$, then $4E(s,2)+E(s,3)\le 3s/2$, so $\di(G)\le 2s$ by Theorem \ref{DetLucHalfn}. If $s=16$, then $\di(S)\le 6$ by Proposition \ref{TransProp}, so $\di(G)\le 4E(16,2)+E(16,3)+6=31$, which gives us what we need. 

Suppose now that $s=6$ or $8$, and let $A_{K}$ be the induced action of $G\cap R^{s}$ on a minimal block. If $A_{K}<GL_{2}(3)$, then $\di(G)\le \max\left\{4E(s,2),3E(s,2)+E(s,3)\right\}+s/2$, which gives the result in each case. So assume that $R=GL_{2}(3)$. Then $G/G\cap SL_{2}(3)^{s}\le 2\wr S$ is transitive of degree $2s$, so by Proposition \ref{TransProp}, $\di(G/G\cap SL_{2}(3)^{s})$ is less than or equal to $4$ if $s=6$ and $6$ if $s=8$. Thus, $\di(G)\le 3E_{sol}(s,2)+E_{sol}(s,3)+s-2$, which gives the result in each case. 

Finally, suppose that $s=2$ or $4$. Then $R/Z$ is a subgroup of $S_{4}$, and if it is intransitive, then it must have order $1$, $2$ or $6$, in which case $\di(G)\le 2E(s,2)+E(s,3)+s/2$ by Corollary \ref{chief} and Theorem \ref{DetLucHalfn}. This yields the result in each case, so assume that $R/Z\le S_{4}$ is transitive. Then $G/G\cap Z^{s}\le \Sym{4}\wr \Sym{(s)}$ is transitive of degree $4s$. Thus, using Proposition \ref{TransProp}, if $s=4$ then $\di(G/G\cap Z^{s})\le 6$, so $\di(G)\le E(4,2)+6=8$, as needed. If $s=2$ and $\di(G/G\cap Z^{s})=4$, then $G/G\cap Z^{s}$, and hence $G$, is a $2$-group, by Proposition \ref{TransProp} so $\di(G)=d(G)\le 2s$ by \cite[Proposition 2.4]{KLS}, and \cite[Theorem 1.2]{derek}. Otherwise, $\di(G/G\cap Z^{s})\le 3$, so $\di(G)\le E(2,2)+3=4$, as needed.

So we may assume that $|F|>3$. Then either $f=2$ and $R$ is metacyclic, or $f=1$ and either $R/Z\le 2^{2}.S_{3}\cong S_{4}$, or $R/Z\le T.2$, for some nonabelian simple group $T$, by Lemmas \ref{FStarO} and \ref{FStarT}. If $G$ is primitive, then $\di(G)\le 3$, since $S_{4}$ and $T$ are both invariably $2$-generated. 

So assume that $G$ is imprimitive. In the case $f=2$ we have $d_{I}(G)\le 2s+s/2+1\le 3s$, using Theorem \ref{DetLucHalfn}. Otherwise, Corollary \ref{chief} yields $\di(G)\le 3E(s,2)+E(s,3)+s+s/2$. It is easy to see that $3E(s,2)+9E(s,3)\le 3s/2$ (from which the result follows), except when $s=2$ or $4$. Furthermore, if $s=2$ or $4$ and $R/Z\le T.2$ for a nonabelian simple group $T$, then $\di(G)\le E(s,2)+1+s/2$, by Corollary \ref{chief} which is less than $3s$ in each case. Thus, writing bars for reduction modulo $Z^{s}$, we may assume that $\overline{R}=R/Z$ has shape $N.X$, where $N$ is elementary abelian of order $4$, and $X=1$, $A_{3}$ or $S_{3}$. Then $\overline{G}/\overline{G}\cap N^{s}$ is either transitive of degree $s$ or $3s$. It follows that $\di(\overline{G}/\overline{G}\cap N^{s})\le s$ in each case, by Proposition \ref{TransProp} (since $s=2$ or $4$). It follows that $\di(G)\le 2E(s,2)+s+s=3s$, as needed, in each of the cases $s=2$ and $s=4$.   

\item $r=3$. We first consider the case $|F|=2$. Then $R\le GL_{3}(2)$ is primitive and irreducible, so either $R=GL_{3}(2)$ is simple or $|R|$ divides $21$. In particular, $\di(G)\le 2$ if $G$ is primitive. So assume that $G$ is imprimitive. Then $\di(G)\le E(s,3)+E(s,7)+s/2+1$, by Corollary \ref{chief} and Theorem \ref{DetLucHalfn}, and the result follows since $E(s,p)\le s/2$ for all primes $p$. 

Assume now that $|F|=3$. Here $R\le GL_{3}(3)$ is primitive and irreducible, so either $R$ has a nonabelian simple normal subgroup $T$ of index at most $2$, or $R$ is soluble of order dividing $78$ (by direct computation). It follows from Lemma \ref{diMin} part (iii) that $\di(G)\le 3$ if $G$ is primitive. Otherwise, Corollary \ref{chief} and Theorem \ref{DetLucHalfn} imply that $\di(G)\le E(s,2)+E(s,3)+E(s,13)+s/2+1\le 2s+1<3s$, as needed (since $E(s,p)\le s/2$ for all primes $p$).

So assume that $|F|>3$. Since $f=1$, either $R/Z\le 3^{2}.Sp_{2}(3)$ or $R/Z\le T.A$, where $T$ is a nonabelian simple group with a projective irreducible representation of degree $3$, and $A\le \Out{(T)}$. By \cite{12,19} and \cite{Atlas}, $|A|\le 3$. Thus, $G$ primitive implies that $\di(G)\le 3$ by Lemmas \ref{diMin} and \ref{Lemma4.1}, and Corollary \ref{KLSSimpleCor}, in either case. So assume that $G$ is imprimitive. In either case we have $\di(G)\le 3E(s,2)+3E(s,3)+s+\di(S)$, by Corollary \ref{chief}. Since $E(s,p)\le s/2$ for all primes $p$, and $E(3,2)=E(3,3)=1$, and $\di(S)\le (s+\delta_{3,s})/2$, the result follows for all $s$.
\end{enumerate}
For the remaining cases below, we may assume that $|F|=2$ or $|F|=3$.\begin{enumerate}\setcounter{enumi}{3}
\item $r=4$. Suppose first that $|F|=2$. Then, since $(r,2^{f}-1)=1$ and $f<r$, Corollary \ref{FStarCor} implies that $R$ is insoluble. Direct computation then implies that $R\le GL_{4}(2)$ is isomorphic to either $A_{5}$, $S_{5}$, $3.S_{5}$, $A_{6}$, $S_{6}$, $A_{7}$, $S_{7}$ or  $L_{4}(2)$. Suppose first that $G$ is primitive. If $G=S_{6}$, then $\di(G)=3$, as needed, so assume otherwise. Then Corollary \ref{KLSSimpleCor} and Proposition \ref{SymmetricGroup} yield $\di(G)\le 2$ in each case, except when $G\cong 3.S_{5}$. But in this case, a $5$-cycle and a $2$-cycle invariably generate $G/M$, where $M\unlhd G$ of order $3$. Clearly any element $x$ of a Sylow $5$-subgroup of $G$ reduces to a $5$-cycle modulo $M$, and centralises $M$. Thus, if $M=\langle z\rangle$, with $|z|=3$, and $y\in G$ with $My$ a $2$-cycle, then $\left\{xz,y\right\}$ invariably generates $G$. 

So assume that $G$ is imprimitive. Corollary \ref{chief} then yields $\di(G)\le E(s,2)+E(s,3)+2+\di(S)$, and the result now follows whenever $s\ge 4$, since $E(s,p)\le s/2$ for all primes $p$, and $\di(S)\le (s+\delta_{s,3})/2$ by Theorem \ref{DetLucHalfn}. If $s=3$, then $\di(G)\le E(3,2)+E(3,3)+2+2=6$, as needed. If $s=2$, then the result follows from using Table 1 and Corollary \ref{chief}, except when $R\cong 3.S_{5}$. But in this case, if $M$ is a minimal normal subgroup of $R$ of order $3$, then since $G$ is large, $G/G\cap M^{2}$ is isomorphic to a transitive subgroup of $S_{5}\wr S_{2}\le S_{10}$. Hence $\di(G/G\cap M^{2})\le 3$ by Proposition \ref{TransProp}. Corollary \ref{chief} then gives $\di(G)\le E(2,3)+3=4$ as needed.

Assume now that $|F|=3$. If $G$ is primitive, then the result follows from Proposition \ref{PrimIrr43Prop}, so assume that $G$ is imprimitive. Then Corollary \ref{chief}, together with Table 1, yields $\di(G)\le \max\left\{8E(s,2)+2E(s,3),7E(s,2)+E(s,5)\right\}+\di(S)$. By using Theorem \ref{DetLucHalfn} (or Proposition \ref{TransProp} when $s=6$), and the definition of the function $E$, it is easy to see that this yields $\di(G)\le 4s$ in all cases, except when $s=2$ or $s=4$.

So assume that $s$ is $2$ or $4$. If $R$ is insoluble, then the result follows easily in each case, using Corollary \ref{chief} and Table 1, except when $R$ has $6$ composition factors of order $2$, and one composition factors isomorphic to $A_{5}$, and $s=2$. But in this case, direct computation quickly shows $R/O_{2}(R)\cong S_{5}$. Thus, $G/G\cap O_{2}(R)^{2}$ is isomorphic to an transitive group of degree $10$, and hence $\di(G/G\cap O_{2}(R))\le 3$, by Proposition \ref{TransProp}. It follows from Corollary \ref{chief} that $\di(G)\le 5E(2,2)+3=8$, as needed.

So assume that $R$ is soluble. Since $f<r$, Corollary \ref{FStarCor} implies that there are two possibilities:\begin{enumerate}[(a)]
\item $f=2$ and $K_{1}/Z\le 2^{2}.S_{3}\cong S_{4}$. Then $\di(G)\le 4E(s,2)+E(s,3)+s+s/2$, by Corollary \ref{chief} and Theorem \ref{DetLucHalfn}. This gives us what we need in each case.
\item $f=1$ and $R/Z$ has shape $N.X$, where $N$ is elementary abelian of order $2^{4}$ and $X\le Sp_{4}(2)\cong S_{6}$ is completely reducible. Recall also that $Y:=A_{K}N/N\le Sp_{4}(2)$ is also completely reducible, since $A_{K}$ a normal subgroup of $R$. By direct computation, the possibilities for $|X|$ and $|Y|$ are $1,3,5,6,9,10,18,20,36$ and $72$. We also compute, for each possible $Y$, the number $j$ of orbits of $Y$ on the nonidentity elements of $N$. Then, by Corollary \ref{chiefpreCor} and Theorem \ref{pmodlemma} Part (iv) we have 
\begin{align}\di(G)\le E(s,2)+(i+\min\left\{j,4\right\})\lfloor bs/\sqrt{\log{s}}\rfloor+kE(s,3)+lE(s,5)+\di(S) \end{align} where $i$, $k$ and $l$ denote the number of composition factors of $Y$ of order $2$, $3$ and $5$, respectively. Apart from two cases, this gives us what we need whenever $s=2$ or $s=4$. 

The two exceptions occur when $|X|=36$ or $72$, and $(|Y|,j)=(36,3)$ or $(72,2)$. In these cases, the result follows from (7.1) when $s=4$, so assume that $s=2$. Then (by direct computation) $X$ has a core-free subgroup of index $6$. Hence, $G/G\cap N^{2}$ is transitive of degree $12$, and hence $\di(\overline{G}/\overline{G}\cap N^{2})\le 4$ by Proposition \ref{TransProp}, where bars denote reduction modulo $Z^{2}$. Also, since $j=3$, Corollary \ref{chiefpreCor} and Theorem \ref{pmodlemma} Part (iv) imply that $d_{\overline{G}}(N)\le 3\lfloor 2b\rfloor=3$. Hence, by Corollary \ref{chief}, we have $\di(G)\le d_{G}(G\cap Z^{2})+$\\ $d_{G}(\overline{G}\cap N^{2})+\di(\overline{G}/\overline{G}\cap N^{2})\le E(s,2)+3\lfloor 2b\rfloor+4=8$, which gives us what we need. \end{enumerate}

\item $r=5$, $7$, $11$, $13$ or $17$. Here, since $f=1$, all characteristic abelian subgroups of $R$ are contained in $Z(R)$, which has order $1$ or $2$, depending on whether $|F|$ has order $2$ or $3$ respectively. Hence, Corollary \ref{FStarCor} implies that $R/Z\le T.A$, where $T$ is a nonabelian simple group with a projective irreducible representation $M$ of degree $r$ over $F$, and $A$ is the subgroup of $\Out{(T)}$ which stabilises $M$. Suppose first that $r=5$ or $r=7$, and $|F|=2$. Then direct computation implies that $R=R/Z\cong L_{r}(2)$. Hence, $\di(G)\le 2$ if $G$ is primitive, and $\di(G)\le 2+s/2+1/2$ by Corollary \ref{chief} and Theorem \ref{DetLucHalfn} if $G$ is imprimitive. This gives the required upper bound in each case.

So we may assume that $(r,|F|)\neq (5,2)$, $(7,2)$. If $G=R$ is primitive, then $G/Z$ is almost simple, so $\di(G)\le 6$, and $\di(G)\le 5$ if $|F|=2$ by Corollary \ref{KLSSimpleCor} and Proposition \ref{dOuterSimple}. So assume that $G$ is imprimitive; we will prove that $\di(G)\le 5s$, which will give us what we need. To see this, Corollary \ref{chief}, Theorem \ref{DetLucHalfn} and Proposition \ref{dOuterSimple} imply that $\di(G)\le 2+3s+E(s,2)+(s+\delta_{s,3})/2$. This yields $\di(G)\le 5s$, since $E(s,2)\le s/2$.

\item $r=6$. Suppose first that $|F|=2$, and that $G$ is primitive. If $G$ is soluble, then since $f<r$, the only possibility is that $f=2$, and that $R\le 3^{1+2}.GL_{2}(3)$. In this case, it is easily checked by direct computation that $\di(G)\le 3$. So assume that $G$ is insoluble. Let $a$ be the number of abelian chief factors of $G\le GL_{6}(2)$, and let $b$ be the number of nonabelian chief factors. By Lemma \ref{diMin}, $\di(G)\le a+2b$, and by direct computation (using the database of irreducible matrix groups in MAGMA), this gives $\di(G)\le 3$, except when $G$'s number $i$ in the MAGMA database is $44,47,52,60,61$ or $62$. Suppose first that $i=47$. Then $G$ has a normal subgroup $N\cong C_{3}$ such that $G/N\cong S_{6}$. Let $P$ be a Sylow $5$-subgroup of $G$, and let $x\in P$ such that $Nx$ is a $5$-cycle in $S_{6}$. Also, let $y$ and $z$ be elements of $G$ which reduce modulo $N$ to a $3$-cycle and a $6$-cycle, respectively, and let $w$ be a generator for $N$. Then, since $x$ centralises $w$, and $Nx,Ny,Nz$ invariably generates $G/N$, we see that $wx,y,z$ invariably generates $G$, as needed. So assume that $i\neq 47$. In each of these cases, $G$ has a subnormal series $1\unlhd N\unlhd G$, in which one of the factors is cyclic, and the other is isomorphic to either $PGL_{2}(7)$, $P\Gamma L_{2}(8)$ or $L_{3}(4).m$ ($m=2$, $3$ or $6$). Thus, $\di(G)\le 3$ by Proposition \ref{SymmetricGroup}.

So assume that $G$ is imprimitive. Using Table 1 and Corollary \ref{chief}, if $R$ is insoluble then we have $\di(G)\le \max\left\{E(s,2)+2E(s,3),E(s,3)+E(s,7)\right\}+2+\di(S)$. Since $E(s,p)$ and $\di(S)$ are bounded above by $s/2$, the result follows. So assume that $R$ is soluble. Then, since $f<6$, Lemma \ref{FStarLemma} implies that $f=2$ and $K_{1}/Z$ has shape $N.X$, where $N$ is elementary abelian of order $3^{2}$, and $X\le Sp_{2}(3)$ is soluble and completely reducible. It follows from Corollary \ref{6CaseCor} that $\di(G)\le 6$ if $s=2$, as needed, and that
\begin{align*}\di(G)\le & E(s,2)+\min\left\{2E(s,3),\lfloor bs/\sqrt{\log{s_{3}}}\rfloor,s/s_{2}\right\}+E(s,2)\\
& +\min\left\{2E(s,2),\lfloor bs/\sqrt{\log{s_{2}}}\rfloor,s/s_{3}\right\}+\delta_{s,8}E(s,2)+(1-\delta_{s,8})E(s,3)+E(s,3)+\di(S)\end{align*} in general. Using the definition of the function $E$, one can easily see that this latter bound, together with Theorem \ref{DetLucHalfn}, yields the result whenever $s\ge 3$.

Next, assume that $|F|=3$. If $G$ is primitive, then $\di(G)\le \lfloor 2\log{6}\rfloor+1=6$ by Theorem \ref{CompRedTheorem}, as needed. Suppose, then, that $G$ is imprimitive. Using Table 1 and Corollary \ref{chief}, we have 
$$\di(G)\le \left\{4E(s,2)+2E(s,3)+E(s,13),4E(s,2)+E(s,3)+E(s,7)+E(s,3)\right\}+\di(S)$$ Using the definition of the function $E$, this bound, together with Theorem \ref{DetLucHalfn} yields the result for all $s$.
\end{enumerate}

For the remaining cases below we assume that $|F|=2$.\begin{enumerate}\setcounter{enumi}{6}
\item $r=8$. Since $f<r$, Lemma \ref{FStarLemma} implies that $R\le GL_{8}(2)$ is insoluble, and using Table 1, together with Corollary \ref{chief}, we have $$\di(G)\le \max\left\{2E(s,2)+E(s,3)+4,E(s,2)+2E(s,3)+4,2E(s,2)+2E(s,3)+E(s,5)+2\right\}+\di(S)$$ when $G$ is imprimitive. Theorem \ref{DetLucHalfn}, together with the bound $E(s,p)\le s/2$ now gives us what we need. 

So we may assume that $G$ is primitive. If $f\ge 4$, then $\di(G)\le 1+2\log{8/f}+1=4$ as needed, so assume that $f=1$ or $f=2$. Suppose first that $f=1$. Then $R=R/Z$ is almost simple, with a projective irreducible representation of degree $4$ or $8$. Hence, $R/\Soc{(R)}$ is cyclic, using the list in \cite[proof of Lemma 4.2]{derek}, and the result follows from Proposition \ref{dOuterSimple} and Corollary \ref{KLSSimpleCor}. 

So we may assume that $f=2$. Then $G\le GL_{4}(4).4$, and $K_{1}=G\cap GL_{4}(4)$ is irreducible and weakly quasiprimitive; let $L$ be the generalised fitting subgroup of $K_{1}$. Using the list in \cite[proof of Lemma 4.2]{derek} and Lemma \ref{FStarLemma}, $L$ has $l$ quasisimple central factors, where $l=1$ or $2$. If $l=1$, then $K_{1}/Z\cong L_{2^{a}}(4)$, $PSp_{2^{a}}(4)$ ($a=1$ or $2$), or $U_{4}(2)$. In particular, $\di(G)\le 1+\di(K_{1}/Z)+1\le 4$ by Corollary \ref{KLSSimpleCor}. So all that remains is the case when $L/Z$ is a direct product of simple groups $T_{1}$, $T_{2}$, where each $T_{i}$ has a projective irreducible representation of degree $2$ over $\mathbb{F}_{4}$. By again using the list in \cite[proof of Lemma 4.2]{derek}, we see that $T_{i}\cong L_{2}(4)\cong A_{5}$ for each $i$. By using the database of irreducible matrix groups in MAGMA, we see that the only possibility for $G\le GL_{8}(2)$ is to have number $j=165,172,185,197$ or $203$ in the database. However, in each of these cases, we find that either $G/L$ is cyclic, or $\di(G/L)=2$ and $Z=1$. Hence, $\di(G)\le 2+\di(L/Z)=4$, in each case by Corollary \ref{KLSSimpleCor}.

\item $r=9$. Again, since we are assuming that $f<r$, Lemma \ref{FStarLemma} implies that $R$ is insoluble (since any $q_{i}$ as in Lemma \ref{FStarLemma} divides $r/f$ and $2^{r/f}-1$). The database of irreducible matrix groups implies that the list of chief factors of $R$ form a sublist of either $[3,7,T]$ or $[2,3,T]$, for a nonabelian simple group $T$, or $[2,L_{2}(7)\times L_{2}(7)]$. Thus, if $G$ is primitive, then $\di(G)\le 4$ by part (iii) of Lemma \ref{diMin}. If $G$ is imprimitive, then Corollary \ref{chief} gives $\di(G)\le E(s,2)+E(s,3)+E(s,7)+2+\di(S)$, and the result again follows easily, using Theorem \ref{DetLucHalfn}, and the fact that $E(s,p)\le s/2$.

\item $r=10$, $14$ or $15$. As in the cases $r=8$ and $r=9$ above, $R$ must be insoluble in each case, by Lemma \ref{FStarLemma}. Now, $r$ divides $uv$, where $u$ and $v$ are prime, so either $K_{1}/Z$ is almost simple, or $Z=C_{1}\cap C_{2}$ where $C_{1}$, $C_{2}$ are two normal subgroups of $K_{1}$, and $K_{1}/C_{i}$ is almost simple, for $i=1$, $2$, whose socles have projective irreducible representations over $F_{1}$ of degree $u$ and $v$ respectively over. Also, $R/K_{1}$ is cyclic of order $1$, $u$, or $v$ (since $f<r$). Note also that, using Corollary \ref{FStarTCors2} and \cite{12,19}, if $T$ is a nonabelian simple group with a projective irreducible representation of degree $d$, and $A$ is as in Lemma \ref{FStarLemma}, then $d\neq 2,3$ if $f=1$; $|A|\le 2$ if $d=2$; $|A|\le 3$ if $d=3$; and $|A|=1$ if $d=5$, using the list in \cite[proof of Lemma 4.2]{derek} ($\ast$). 

Suppose first that $r=10$. By ($\ast$), $K_{1}/Z$ is an almost simple group, whose socle has a projective irreducible representation of degree $2$ or $5$ over $F_{1}$. Hence, if $G$ is primitive then $\di(K_{1}/Z)\le 3$ by ($\ast$). Thus, $\di(G)\le \di(G/K_{1})+\di(K_{1}/Z)+\di(Z)\le 5$, as needed. Otherwise, Corollary \ref{chief} and ($\ast$) imply that $\di(G)\le \max\left\{E(s,2),E(s,5)\right\}+E(s,2)+s+2+s+\di(S)$. This is less than $5s$, using Theorem \ref{DetLucHalfn}, and the bound $E(s,p)\le s/2$.

Assume now that $r=14$ or $r=15$. It follows, from the first paragraph above and Corollary \ref{chief}, that $K_{1}/Z$ is almost simple. Hence, $\di(G)\le 1+\di(K_{1}/Z)+1\le 7$ if $G$ is primitive. So we my assume that $G$ is imprimitive. Then $\di(G)\le \max\left\{E(s,u),E(s,v)\right\}+E(s,2)+2s+2+s+\di(S)$, by Corollary \ref{chief}. By using the bound $E(s,p)\le s/2$, and the bound of Theorem \ref{DetLucHalfn}, the result now follows.

\item $r=12$. Suppose first that $K_{1}$ has a $q$-core not contained in $Z$, for some prime $q$. Then $q$ divides $r/f$ and $2^{f}-1$, so $f=2$ or $f=4$. Thus, $r/f$ divides $6$, so $q=3$ is the only possibility. Hence, using Lemma \ref{FStarLemma}, we have $K_{1}/C_{K_{1}}(O_{3}(K_{1}))\le 3^{2}.Sp_{2}(3)$. The generalised fitting subgroup $L$ of $K_{1}$ has at most one another central factor, and if it has such a central factor $U$, then $f=2$ and $U$ is insoluble, with $K_{1}/C_{K_{1}}(U)\le T.2$, for a nonabelian simple group $T$, by Lemmas \ref{FStarO} and \ref{FStarT}. In this case, $T$ has a projective irreducible representation of degree $2$ over $\mathbb{F}_{4}$, so we must have $T=A_{5}$. Hence, $\di(K_{1}/C_{K_{1}}(U))=2$ by Proposition \ref{SymmetricGroup}, and the result now follows if $G$ is primitive, since $Z=C_{K_{1}}(O_{3}(K_{1}))\cap C_{K_{1}}(T)$, $\di(G/K_{1})\le 1$ and every subnormal subgroup of $2^{1+2}.Sp_{2}(3)$ is invariably $3$-generated by direct computation. If $G$ is imprimitive, then Corollary \ref{chief} implies that $$\di(G)\le 5E(s,2)+4E(s,3)+E(s,5)+\di(S)$$ The result now follows, since $E(s,p)\le s/2$, and $\di(S)\le (s+\delta_{s,3})/2$. 

So we may assume that all central factors of $L$ are quasisimple. If $f\ge 4$, then $\di(G)\le 5s+\di(S)$ by Part 2 of Theorem \ref{CompRedTheorem}, so we may assume that $f=1$, $2$ or $3$ (using Theorem \ref{DetLucHalfn}. Then $K_{1}$ has normal subgroups $C_{1}$, $\hdots$, $C_{t}$, with $t\le 3$, such that $K_{1}/C_{i}$ has shape specified in the following table. Furthermore, in each case, $K_{1}/C_{i}$ is either almost simple with a projective irreducible representation over $\mathbb{F}_{2^{f}}$ of degree $t_{i}$ dividing $r/f$, or $f=3$, $t=1$, and $K_{1}/C_{1}$ has shape $T_{1}^{2}.D_{8}$, where $T_{1}$ is a nonabelian simple group. To list all of the possibilities for the groups $K_{1}/C_{i}$, we use the list in \cite[proof of Lemma 4.2]{derek}. The possibilities are as follows: 
 \begin{center}\begin{tabular}{|c|c|c|c|p{10cm}|}
\hline
  \multicolumn{5}{ |c| }{Table 2}\\
  \hline
 $f$ & $K_{1}/C_{1}$ & $K_{1}/C_{2}$ & $K_{1}/C_{3}$ & \\
  \hline\hline
 $1$ & $A$ & $L_{3}(2).2$ & - & $A\cong A_{m}$, $S_{m}$, or $L_{4}(2)$, where $5\le m\le 7$. \\
 \hline
 $1$ & $T_{1}.A_{1}$ & - & - & $T_{1}$ is a nonabelian simple group, and $A_{1}\le \Out{(T_{1})}$.\\
 \hline
 $2$ & $A$ & $A'$ & - & $A\cong A_{5}$ or $S_{5}$; $A'\cong L_{3}(4)$, $PGL_{3}(4)$ or $A_{6}$. \\
 \hline
 $2$ & $A$ & - & - & $A\cong L_{6}(4).3=PGL_{6}(4)$, $PSp_{6}(4)$, $U_{6}(3).2$, $A_{7}$, $M_{22}$, $U_{4}(3).2$, $L_{4}(4).2$, $U_{4}(4).2$, $G_{2}(4)$, $J_{2}$, $L_{2}(13)$.\\
 \hline
 $3$ & $L_{2}(8)$ & $L_{2}(8)$ & - & $T_{i}$ is a nonabelian simple group. \\
 \hline
 $3$ & $T_{1}^{2}.D_{8}$ & - & - & $T_{i}$ is a nonabelian simple group. \\
 \hline
 $3$ & $T_{1}.A_{1}$ & - & - & $T_{1}$ is a nonabelian simple group, and $A_{1}\le \Out{(T_{1})}$ of order at most $4$.\\
 \hline
\end{tabular}\end{center}
Thus, if $G$ is primitive, then $\di(G/K_{1}), \di(K_{1}/Z)\le 1-\delta_{f,1}$, and going through each of the cases in Table 2 above, and applying Corollary \ref{KLSSimple} and Propositions \ref{dOuterSimple} and \ref{SymmetricGroup} where necessary, we get $\di(G)\le 6$. If $G$ is imprimitive using Corollary \ref{chief} we get $\di(G)\le 2E(s,2)+2s+4+\max\left\{E(s,2),E(s,3)\right\}+\di(S)$, which is less than or equal to $6s$, using the bound $E(s,p)\le s/2$ and Theorem \ref{DetLucHalfn}.

\item $r=16$. If $G$ is primitive, then $\di(G)\le 2\log{16}=8$ by Part 2 of Theorem \ref{CompRedTheorem}, so assume that $G$ is imprimitive. If $f>1$, then Part 2 of Theorem \ref{CompRedTheorem} implies that $\di(G)\le 7s+\di(S)$ which is less than $8s$, by Theorem \ref{DetLucHalfn}. So assume that $f=1$. Then each central factor $U_{j}$ of the generalised Fitting subgroup $L$ of $R$ is insoluble; let $T_{j}$, $S_{j}$, $t_{j}$, and $s_{j}$ be as in Lemma \ref{FStarT}, for $1\le j\le l$. Then $l\le 4$, $\prod_{i=1}^{l}s_{j}^{t_{j}}$ divides $16$, and $R/C_{R}(T_{j})\le S_{j}^{t_{j}}.(A_{j}\wr \Sym{(t_{j})})$, where $A_{j}\le \Out{(S_{j})}$. Furthermore, $|A_{j}|\le 2$ if $s_{j}=2$. Assume first that $R=R/Z$ has shape $T^{4}.(2\wr X)$, where $X\le S_{4}$ is the induced action of $R$ on the four direct factors in $T^{4}$. If $X$ is intransitive, then $|X|$ has order $1$, $2$, $3$ or $6$. If $X$ is transitive, then $T^{4}$ is a minimal normal subgroup of $R$, so $T^{4}$ is a chief factor of $R$. Hence, Corollary \ref{chief} implies that $\di(G)\le \max\left\{7E(s,2)+E(s,3)+2,5E(s,2)+E(s,3)+8\right\}+\di(S)$, and this gives $\di(G)\le 8s$, using Theorem \ref{DetLucHalfn}, and the bound $E(s,p)\le s/2$.

Going through each of the remaining possibilities for the pairs $(s_{j},t_{j})$ (as in the case $r=12$ above), and applying Corollary \ref{chief}, we get $\di(G)\le \max\left\{4E(s,2)+2s+6,2E(s,2)+4s+4\right\}$\\$+\di(S)$, which gives us what we need, using the bounds from Theorem \ref{DetLucHalfn}, and the bound $E(s,p)\le s/2$. 
\end{enumerate}\end{proof}

\section{The proof of Theorems \ref{TransInvTheorem} and \ref{PrimInvTheorem}}\label{ProofSection2}
Throughout the remainder of the paper, we will make use of the Vinogradov notation defined in Section 1: recall, $A\ll B$ means that $A=O(B)$. We begin with a useful observation.
\begin{Lemma}\label{UsefulBoundGar} Suppose that $n=ab$, with $a$ and $b$ at least $2$. Then
$$\frac{b\log{a}}{\sqrt{\log{b}}}\ll \frac{ab}{\sqrt{\log{ab}}}$$
where the implied constant is independent of $a$, $b$, and $n$.\end{Lemma}
\begin{proof} The monotonicity of the function $x/\sqrt{\log{x}}$ implies that 
$$\frac{b\log{a}}{\sqrt{\log{b}}}\le b\log{a}\le \frac{ba}{\sqrt{\log{a}}}.$$
Thus, we have 
$$\frac{b\log{a}}{\sqrt{\log{b}}}\le \frac{ba}{\sqrt{\log{\max\{a,b\}}}}=\frac{n}{\sqrt{\log{\max\{a,b\}}}}.$$
Since $ab=n$, either $a$ or $b$ must be greater than or equal to $\sqrt{n}$. The result follows.\end{proof} 
 
\begin{proof}[Proof of Theorem \ref{TransInvTheorem}] We will prove the theorem by induction on $n$. For the initial step, assume that $G$ is primitive. Then by Theorems \ref{Pyb} and \ref{MinTheoremInv}, we have $\di(G)\le a(G)\ll \log{n}\ll n/\sqrt{\log{n}}$, as required.

The inductive step concerns imprimitive $G$. Then $G$ is a large subgroup in a wreath product $R\wr S$, where $R$ is a primitive permutation group of degree $r\ge 2$, $S$ is a transitive permutation group of degree $s\ge 2$, and $rs=n$. In particular, $a(R)\ll\log{r}$ by Theorem \ref{Pyb}, and $\di(S)\ll s/\sqrt{\log{s}}$ by the inductive hypothesis. Hence, by Corollary \ref{37} we have
\begin{align}\label{GMLT1}\di(G) &\ll \frac{a(R)s}{\sqrt{\log{s}}}\ll \frac{s\log{r}}{\sqrt{\log{s}}}.\end{align}
The result now follows immediately from Lemma \ref{UsefulBoundGar}.\end{proof}

\begin{Proposition}\label{AffineCase} Let $G\le GL_m(p)$ be finite and irreducible. Then
\begin{align*}
\di(G) &\ll \frac{\log{p^{m}}}{\sqrt{\log{\log{p^m}}}}.\end{align*}\end{Proposition}
\begin{proof} The proof here follows the same strategy as the proof of Theorem \ref{TransInvTheorem} above. Suppose first that $G$ is primitive. Then $\di(G)\le a( 2\log{m}+1$ by Theorem \ref{CompRedTheorem} Part(ii). Since $m\le \log{p^{m}}$, it follows that 
$$\di(G)\ll \log{\log{p^m}}\le \frac{\log{p^{m}}}{\sqrt{\log{\log{p^m}}}}$$
as needed.

So we may assume that $G$ is imprimitive. Thus, $G$ is a large subgroup in a wreath product $R\wr S$, where $R\le GL_r(p)$ is primitive, $S\le \Sym(s)$ is transitive of degree $s\geq 2$, and $rs=n$. Since $a(R)\ll\log{p^r}$ by Theorem \ref{Pyb2}, and $\di(S)\ll s/\sqrt{\log{s}}$ by Theorem \ref{TransInvTheorem}, Corollary \ref{37}
 yields
\begin{align}\label{GMLT11}\di(G) &\ll \frac{a(R)s}{\sqrt{\log{s}}}\ll \frac{s\log{p^r}}{\sqrt{\log{s}}}.\end{align}
As in the proof of Theorem \ref{TransInvTheorem} above, the result now follows immediately from Lemma \ref{UsefulBoundGar}.\end{proof}

\begin{proof}[Proof of Theorem \ref{PrimInvTheorem}] We will consider each of the cases of the O'Nan-Scott Theorem for primitive permutation groups; the form of the theorem we use is from \cite{ONanScott}.
\begin{enumerate}[(I)]
\item $G$ is a subgroup in the affine general linear group $AGL_{m}(p)$, and $n=p^{m}$, $p$ prime. Here, $G$ has a unique minimal normal subgroup $B$, which is elementary abelian of order $p^{m}$, and $G/B$ is isomorphic to an irreducible subgroup of $GL_{m}(p)$. The result now follows from Proposition \ref{AffineCase} and Lemma \ref{diMin} Part (ii).
\item $G$ is almost simple. Then $\di(G)\le 5$ by Corollary \ref{KLSSimpleCorOut}.
\item \begin{enumerate}[(a)] 
\item Simple diagonal action. Here, $n=|T|^{k-1}$, where $T$ is a non-abelian finite simple group, and $k\ge 2$. Furthermore, $B:=\Soc{(G)}\cong T^{k}$, and if $P\le \Sym{(k)}$ is the induced action of $G$ on on the direct factors of $B$, then one of the following holds:\begin{enumerate}[(i)]
\item $P$ is primitive, $B$ is the unique minimal normal subgroup of $G$, and $G/B$ has shape $E.P$ where $E\le \Out(T)$, or;
\item $k=2$, $P=1$, and $G\cong B$.\end{enumerate} 
Suppose first that case (i) holds. Then since any subgroup of $\Out(T)$ is invariably $3$-generated, Lemma \ref{diMin} and Theorem \ref{TransInvTheorem} yields $$\di(G)\le 2+\di(G/B)\ll 5+\frac{k}{\sqrt{\log{k}}}\ll \frac{k}{\sqrt{\log{k}}}.$$ Since $k\le \log{n}$ and $k/\sqrt{\log{k}}$ is an increasing function, the result follows.   

In the second case, $G\cong T^{2}$, so $\di(G)\le 3$ by Corollary \ref{KLSSimpleCor}, and the result again follows.
\item Product action. Let $R\le \Sym(r)$ be a primitive permutation group of type (II) or (III)(a), and let $S$ be a transitive permutation group of degree $s$. Then, with the product action, $G$ is a large subgroup of the wreath product $R\wr S$. In particular, $n=r^{s}$. Hence, we have
\begin{align}\label{Gar1} \di(G)\ll \frac{a{(R)}s}{\sqrt{\log{s}}}+2\nab{(R)}+\di(S).\end{align}
by Corollary \ref{37}.

Now, by Proposition \ref{OutPermProp}, $\ab{(R)}\ll \log{\log{r}}$ and $\nab{(R)}=1$ if $R$ is of type (II). If $R$ is of type (III)(a)(i), then adopting the same notation as used in that case above, we have $\ab{(R)}=a{(R/B)}$, and $R/B\le \Out{(T)}\times P$ projects onto the primitive group $P$ of degree $k=\log_{|T|}(r)+1\ll \log{r}$. Then 
\begin{align*}\ab(R)=\ab{(R/B)}\le \log{|\Out{(T)}|}+\ab{(P)} &\ll \log{\log{r}}+\log{k}\ll \log{\log{r}}\end{align*}
and 
\begin{align*}\nab(R)=1+\nab{(R/B)}= 1+\nab{(P)} &\ll 1+\log{k}\ll \log{\log{r}}\end{align*}
by Proposition \ref{OutPermProp} and Theorem \ref{Pyb}. Finally, $\ab{(R)}=0$ and $\nab{(R)}=2$ if $R$ is of type (III)(a)(ii).

Thus, by (\ref{Gar1}) and Theorem \ref{TransInvTheorem} we have
\begin{align}\label{Gar2} \di(G)\ll \frac{s\log{\log{r}}}{\sqrt{\log{s}}}+\frac{s}{\sqrt{\log{s}}}\ll \frac{s\log{\log{r}}}{\sqrt{\log{s}}}.\end{align}
Let $x:=\log{r}$. It now follows immediately from Lemma \ref{UsefulBoundGar} that 
$$\di(G)\ll \frac{{xs}}{\sqrt{\log{xs}}}=\frac{{\log{r^s}}}{\sqrt{\log{\log{r^s}}}}$$
as needed.
\item Twisted wreath action. Here, $G$ is a semidirect product $T^{s}\rtimes S$, where $S$ is a transitive permutation group of degree $s\geq 1$, and $T$ is a non-abelian simple group. Furthermore, $n=|T|^{s}$, and $T^{s}$ is the unique minimal normal subgroup of $G$. If $s=1$, then $\di(G)=\di(T)\le 2$ by Corollary \ref{KLSSimpleCor}, so assume that $s\ge 2$. Lemma \ref{diMin} and Theorem \ref{TransInvTheorem} then give
$$\di(G)\ll\di(G/T^s)\ll\frac{s}{\sqrt{\log{s}}}\le \frac{\log{|T|^{s}}}{\sqrt{\log{\log{|T|^{s}}}}}$$
and the proof is complete.\end{enumerate}\end{enumerate}\end{proof}

\end{document}